\documentclass{amsart}
\usepackage{amsfonts}
\usepackage{amssymb}
\usepackage{amsmath}
\usepackage{cite}
\usepackage{graphics}
\usepackage{latexsym}
\newtheorem{theorem}{Theorem}[section]
\theoremstyle{plain}

\newtheorem{lemma}{Lemma}[section]

\numberwithin{equation}{section}

\begin{document}
\title[Geometric Properties of function $N_\nu(z)$]{Geometric Properties of
function $az^{2}J_{\nu }^{\prime \prime }(z)+bzJ_{\nu }^{\prime
}(z)+cJ_{\nu }(z)$}
\author{Sercan Kaz\i mo\u glu} \author{Erhan Deniz} 
\address{Department of Mathematics, Faculty of Science and Letters, Kafkas
University, Kars, Turkey}
\email{srcnkzmglu@gmail.com (S. Kaz\i mo\u glu)}
\email{edeniz36@gmail.com (E. Deniz)} 

\begin{abstract}
In this paper our aim is to find the radii of starlikeness and convexity for
three different kind of normalization of the $N_\nu(z)=az^{2}J_{\nu }^{\prime \prime }(z)+bzJ_{\nu }^{\prime
}(z)+cJ_{\nu }(z)$ function, where $J_\nu(z)$ is called the Bessel function of the first kind of order $\nu.$ The key tools in the proof of our
main results are the Mittag-Leffler expansion for $N_\nu(z)$ function and
properties of real zeros of it. In addition, by using the Euler-Rayleigh
inequalities we obtain some tight lower and upper bounds for the radii of
starlikeness and convexity of order zero for the normalized $N_\nu(z)$
function. Finally, we evaluate certain multiple sums of the zeros for $N_\nu(z)$ function.
\end{abstract}

\maketitle

\section{Introduction}

Denote by $\mathbb{D}_{r}=\left\{ {z\in \mathbb{C}:\;\left\vert z\right\vert
<r}\right\} \quad (r>0)$ the disk of radius $r$ and let $\mathbb{D}=\mathbb{D%
}_{1}.$ Let $\mathcal{A}$ be the class of analytic functions $f$ in the open
unit disk $\mathbb{D}$ which satisfy the usual normalization conditions $%
f(0)={f}^{\prime }(0)-1=0.$ Traditionally, the subclass of $\mathcal{A}$
consisting of univalent functions is denoted by $\mathcal{S}.$ We say that
the function $f$ $\in \mathcal{A}$ is starlike in the disk $\mathbb{D}_{r}$
if $f$ is univalent in $\mathbb{D}_{r}$, and $f(\mathbb{D}_{r})$ is a
starlike domain in $%
%TCIMACRO{\U{2102} }%
%BeginExpansion
\mathbb{C}
%EndExpansion
$ with respect to the origin. Analytically, the function $f$ is starlike in $%
\mathbb{D}_{r}$ if and only if $\operatorname{Re}\left( \frac{zf^{\prime }(z)}{f(z)}%
\right) >0,$ $z\in \mathbb{D}_{r}.$ For $\beta \in \lbrack 0,1)$ we say that
the function $f$ is starlike of order $\beta $ in $\mathbb{D}_{r}$ if and
only if $\operatorname{Re}\left( \frac{zf^{\prime }(z)}{f(z)}\right) >\beta ,$ $z\in 
\mathbb{D}_{r}.$ We define by the real number
\begin{equation*}
r_{\beta }^{\ast }(f)=\sup \left\{ r\in \left( 0,r_{f}\right) :\operatorname{Re}%
\left( \frac{zf^{\prime }(z)}{f(z)}\right) >\beta \;\text{for all }z\in 
\mathbb{D}_{r}\right\}
\end{equation*}%
the radius of starlikeness of order $\beta $ of the function $f$. Note that $%
r^{\ast }(f)=r_{0}^{\ast }(f)$ is the largest radius such that the image
region $f(\mathbb{D}_{r_{\beta }^{\ast }(f)})$ is a starlike domain with
respect to the origin.

The function $f$ $\in \mathcal{A}$ is convex in the disk $\mathbb{D}_{r}$ if 
$f$ is univalent in $\mathbb{D}_{r}$, and $f(\mathbb{D}_{r})$ is a convex
domain in $%
%TCIMACRO{\U{2102} }%
%BeginExpansion
\mathbb{C}
%EndExpansion
$. Analytically, the function $f$ is convex in $\mathbb{D}_{r}$ if and only
if $\operatorname{Re}\left( 1+\frac{zf^{\prime \prime }(z)}{f^{\prime }(z)}\right)
>0,$ $z\in \mathbb{D}_{r}.$ For $\beta \in \lbrack 0,1)$ we say that the
function $f$ is convex of order $\beta $ in $\mathbb{D}_{r}$ if and only if $%
\operatorname{Re}\left( 1+\frac{zf^{\prime \prime }(z)}{f^{\prime }(z)}\right)
>\beta ,$ $z\in \mathbb{D}_{r}.$ The radius of convexity of order $\beta $
of the function $f$ is defined by the real number
\begin{equation*}
r_{\beta }^{c}(f)=\sup \left\{ r\in \left( 0,r_{f}\right) :\operatorname{Re}\left( 1+%
\frac{zf^{\prime \prime }(z)}{f^{\prime }(z)}\right) >\beta \;\text{for all }%
z\in \mathbb{D}_{r}\right\} .
\end{equation*}%
Note that $r^{c}(f)=r_{0}^{c}(f)$ is the largest radius such that the image
region $f(\mathbb{D}_{r_{\beta }^{c}(f)})$ is a convex domain.

The Bessel function of the first kind of order $\nu $ is defined by \cite[p.
217]{Olver} 
\begin{equation}
J_{\nu }(z)=\sum_{n=0}^{\infty }\frac{\left( -1\right) ^{n}}{n!\Gamma (n+\nu
+1)}\left( \frac{z}{2}\right) ^{2n+\nu }\text{ \ \ }\left( z\in 
%TCIMACRO{\U{2102} }%
%BeginExpansion
\mathbb{C}
%EndExpansion
\right) .  \label{J1}
\end{equation}%
We know that it has all its zeros real for $\nu >-1$. Here now we consider
mainly the general function%
\begin{equation*}
N_{\nu }(z)=az^{2}J_{\nu }^{\prime \prime }(z)+bzJ_{\nu }^{\prime
}(z)+cJ_{\nu }(z)
\end{equation*}%
studied by Mercer \cite{Mercer}. Here, as in \cite{Mercer}, $q=b-a$ and 
\begin{equation*}
\left(c=0 \text{ and } q\neq0 \right) \text{ or } \left(c>0 \text{ and } q>0
\right) .
\end{equation*}
From (\ref{J1}), we have the power series representation%
\begin{equation}
N_{\nu }(z)=\sum_{n=0}^{\infty }\frac{Q(2n+\nu )\left( -1\right) ^{n}}{%
n!\Gamma (n+\nu +1)}\left( \frac{z}{2}\right) ^{2n+\nu }\text{ \ \ }\left(
z\in 
%TCIMACRO{\U{2102} }%
%BeginExpansion
\mathbb{C}
%EndExpansion
\right)  \label{J2}
\end{equation}%
where $Q(\nu)=a\nu(\nu-1)+b\nu+c$ $\left( a,b,c\in 
%TCIMACRO{\U{211d} }%
%BeginExpansion
\mathbb{R}
%EndExpansion
\right) .$ There are three important works on the function $N_{\nu }.$ First
Mercer's paper \cite{Mercer} which it has been proved that the $k$\textit{th}
positive zero of $N_{\nu }$ increases with $\nu $ in $\nu >0$. Second,
Ismail and Muldoon \cite{IS} showed that under the conditions $a,b,c\in 
%TCIMACRO{\U{211d} }%
%BeginExpansion
\mathbb{R}
%EndExpansion
$ such that $c=0$ and $b\neq a$ or $c>0$ and $b>a$;
\begin{itemize}
\item[(i)] For $\nu > 0$, the zeros of $N_{\nu }(z)$ are either real or
purely imaginary.

\item[(ii)] For $\nu \geq \max \{0,\nu _{0}\}$, where $\nu _{0}$ is the
largest real root of the quadratic $Q(\nu )=a\nu (\nu -1)+b\nu +c,$ the the
zeros of $N_{\nu }(z)$ are real.

\item[(iii)] If $\nu >0$, $\left( a\nu ^{2}+(b-a)\nu +c\right) \diagup
(b-a)>0$ and $a\diagup (b-a)<0$, the zeros of $N_{\nu }(z)$ are all real
except for a single pair which are conjugate purely imaginary.
\end{itemize}
Lastly, Baricz, \c{C}a\u{g}lar and Deniz \cite{Erhandeniz} obtained sufficient and necessary
conditions for the starlikeness of a normalized form of $N_{\nu }$ by using
results of Mercer \cite{Mercer}, Ismail and Muldoon \cite{IS} and Shah and Trimble \cite{Shah}. In this
paper, we deal with the radii of starlikeness and convexity of order $\beta $
for the functions $f_{\nu }(z),\;g_{\nu }(z)$ and $h_{\nu }(z)$ defined by (%
\ref{F1}) in the case when $\nu \geq \max \{0,\nu _{0}\}$. Also we
determined tight lower and upper bounds for the radii of starlikeness and
convexity of these functions. The key tools in their proofs were some new
Mittag-Leffler expansions for quotients of the function $N_{\nu }$, special
properties of the zeros of the function $N_{\nu }$ and their derivatives,
Euler-Rayleigh inequalities and the fact that the smallest positive zeros of
some Dini functions are less than the first positive zero of the function $
N_{\nu }$. Recently, for studies on the Geometric properties of Bessel functions, see \cite{Baricz,Ba1,Ba3,Erhandeniz,Ba21,Br,Ca,De,Deniztrev,Kr,Sza}

Note that $N_{\nu }$ is not belongs to $\mathcal{A}$. To prove the main
results we need normalizations of the function $N_{\nu }$. In this paper we
focus on the following normalized forms%
\begin{eqnarray}
f_{\nu }(z) &=&\left[ \frac{2^{\nu }\Gamma (\nu +1)}{Q(\nu)}N_{\nu }(z)%
\right] ^{\frac{1}{\nu }},  \label{F1} \\
g_{\nu }(z) &=&\frac{2^{\nu }\Gamma (\nu +1)z^{1-\nu }}{Q(\nu)}N_{\nu }(z), 
\notag \\
h_{\nu }(z) &=&\frac{2^{\nu }\Gamma (\nu +1)z^{1-\frac{\nu }{2}}}{Q(\nu)}%
N_{\nu }(\sqrt{z}).  \notag
\end{eqnarray}

In the rest of this paper, the quadratic $Q(\nu )=a\nu (\nu -1)+b\nu +c$ will always  provide on $a,b,c\in \mathbb{R} $ $\left(c=0 \text{ \  and \ } a\neq b    \right)$ or $\left(c> 0 \text{ \  and \ } a< b    \right).$ Moreover, $\nu_0$ is the largest real root of the quadratic $Q(\nu )$ defined according to the above conditions.

\subsection{Zeros of hyperbolic polynomials and the Laguerre--P\'{o}lya
class of entire functions}

In this subsection, we recall some necessary information about polynomials
and entire functions with real zeros. An algebraic polynomial is called
hyperbolic if all its zeros are real. We formulate the following specific
statement that we shall need, see \cite{Deniztrev} for more details.

By definition, a real entire function $\psi $ belongs to the Laguerre--P\'{o}%
lya class $\mathcal{LP}$ if it can be represented in the form%
\begin{equation*}
\psi (x)=cx^{m}e^{-ax^{2}+\beta x}\prod\limits_{k\geq 1}\left( 1+\frac{x}{%
x_{k}}\right) e^{-\frac{x}{x_{k}}},
\end{equation*}%
with $c,\beta ,x_{k}\in 
%TCIMACRO{\U{211d} }%
%BeginExpansion
\mathbb{R}
%EndExpansion
,$ $a\geq 0,$ $m\in 
%TCIMACRO{\U{2115} }%
%BeginExpansion
\mathbb{N}
%EndExpansion
\cup \{0\}$ and $\sum x_{k}^{-2}<\infty .$ Similarly, $\phi $ is said to be
of type $\mathcal{I}$ in the Laguerre-P\'{o}lya class, written $\varphi \in 
\mathcal{LPI}$, if $\phi (x)$ or $\phi (-x)$ can be represented as%
\begin{equation*}
\phi (x)=cx^{m}e^{\sigma x}\prod\limits_{k\geq 1}\left( 1+\frac{x}{x_{k}}%
\right) ,
\end{equation*}%
with $c\in 
%TCIMACRO{\U{211d} }%
%BeginExpansion
\mathbb{R}
%EndExpansion
,$ $\sigma \geq 0,$ $m\in 
%TCIMACRO{\U{2115} }%
%BeginExpansion
\mathbb{N}
%EndExpansion
\cup \{0\},$ $x_{k}>0$ and $\sum x_{k}^{-1}<\infty .$ The class $\mathcal{LP}
$ is the complement of the space of hyperbolic polynomials in the topology
induced by the uniform convergence on the compact sets of the complex plane
while $\mathcal{LPI}$ is the complement of the hyperbolic polynomials whose
zeros possess a preassigned constant sign. Given an entire function $\varphi 
$ with the Maclaurin expansion%
\begin{equation*}
\varphi (x)=\sum_{k\geq 0}\mu _{k}\frac{x^{k}}{k!},
\end{equation*}%
its Jensen polynomials are defined by 
\begin{equation*}
P_{m}(\varphi ;x)=P_{m}(x)=\sum_{k=0}^{m}\left( 
\begin{array}{c}
m \\ 
k%
\end{array}%
\right) \mu _{k}x^{k}.
\end{equation*}%
The next result of Jensen \cite{Je} is a well-known characterization of
functions belonging to $\mathcal{LP}$.

\begin{lemma}
\label{Le2} The function $\varphi $ belongs to $\mathcal{LP}$ $(\mathcal{LPI}
$, respectively$)$ if and only if all the polynomials $P_{m}(\varphi ;x)$, $%
m=1,2,...,$ are hyperbolic (hyperbolic with zeros of equal sign). Moreover,
the sequence $P_{m}(\varphi ;z\diagup n)$ converges locally uniformly to $%
\varphi (z)$.
\end{lemma}

The following result is a key tool in the proof of main results.

\begin{lemma}
\label{Le4} If $\nu \geq \max\{0,\nu_0\}$ then the functions $z\longmapsto
\Psi_\nu(z)=\frac{2^{\nu }\Gamma (\nu +1)}{Q(\nu )z^{\nu }}N_\nu(z)$ has
infinitely many zeros and all of them are positive. Denoting by $%
\lambda_{\nu,n}$ the $n$th positive zero of $\Psi_\nu(z),$ under the same
conditions the Weierstrassian decomposition 
\begin{equation*}
\Psi_\nu(z)=\prod\limits_{n\geq 1}\left( 1-\frac{z^{2}}{ \lambda_{\nu ,n}
^{2}}\right)
\end{equation*}
is valid, and this product is uniformly convergent on compact subsets of the
complex plane. Moreover, if we denote by $\lambda_{\nu ,n}^\prime$ the $n$th
positive zero of $\Phi_\nu^\prime(z),$ where $\Phi_\nu(z)=z^\nu \Psi_\nu(z), 
$ then the positive zeros of $\Psi_\nu(z)$ are interlaced with those of $%
\Phi_\nu^\prime(z).$ In the other words, the zeros satisfy the chain of
inequalities 
\begin{equation*}
\lambda_{\nu ,1}^\prime<\lambda_{\nu ,1}<\lambda_{\nu
,2}^{\prime}<\lambda_{\nu ,2}<\lambda_{\nu ,3}^\prime<\lambda_{\nu
,3}<\cdots .
\end{equation*}
\end{lemma}

\begin{proof}
Recall that \cite{IS} proved that if $\nu \geq \max\{0,\nu_0\}$ then $%
N_\nu(z)$ entire function infinitely many positive zeros. Since the function 
$N_\nu(z)$ is entire, its infinite product was presented by erhan
\cite{Erhandeniz}.Using the infinite product representation we get that 
\begin{equation}  \label{lemd1}
\frac{\Phi_\nu^\prime(z)}{\Phi_\nu(z)}=\frac{\nu}{z}+\frac{\Psi_\nu^\prime(z)%
}{\Psi_\nu(z)}=\frac{\nu}{z}+\sum_{n\geq 1}\frac{2z}{z^2-\lambda_{\nu,n }^2}
\end{equation}
Differentiating both sides of (\ref{lemd1}), we have 
\begin{equation*}
\frac{d}{dz}\left( \frac{\Phi_\nu^\prime(z)}{\Phi_\nu(z)} \right)=- \frac{\nu%
}{z^2}-2\sum_{n\geq 1}\frac{z^2+\lambda_{\nu ,n}^2}{\left(z^2-\lambda_{\nu,n
}^2\right)^2},~~~ z\neq \lambda_{\nu ,n}. 
\end{equation*}
The right hand side of the above expression is real and negative for each $z$
real, $\nu \geq \max\{0,\nu_0\}.$ Thus, the quotient on the left side of (%
\ref{lemd1}) is a strictly decreasing function from $+\infty$ to $-\infty$
as $z$ increases through real values over the open interval $%
\left(\lambda_{\nu ,n},\lambda_{\nu ,n+1} \right),~~ n\in \mathbb{N}.$ Hence
the function $z\longmapsto \Phi_\nu^\prime(z)$ vanishes just once between
two consecutive zeros of the function $z\longmapsto \Phi_\nu(z).$
\end{proof}

\subsection{Euler-Rayleigh Sums for Positive Zeros of $N_{\protect\nu }(z)$}

Baricz, \c{C}a\u{g}lar and Deniz \cite{Erhandeniz} proved Mittag-Leffler expansion of $N_{\nu }(z)$ as
follows 
\begin{equation}
N_{\nu }(z)=\alpha z^2J^{\prime\prime}_\nu(z)+bzJ^\prime_\nu(z)+cJ_\nu(z)= 
\frac{Q(\nu )z^{\nu }}{2^{\nu }\Gamma (\nu +1)}\prod\limits_{n\geq 1}\left(
1-\frac{z^{2}}{ \lambda_{\nu ,n} ^{2}}\right)  \label{r11}
\end{equation}%
where $Q(\nu)=a\nu\left(\nu-1 \right)+b\nu+c,~~~ \left(a,b,c\in \mathbb{R}%
\right)$ and $\lambda_{\nu ,n}$ is the $n$th positive zero of $N_{\nu }(z)$  
$(n\in 
%TCIMACRO{\U{2115} }%
%BeginExpansion
\mathbb{N}
%EndExpansion
)$. Therefore we can write 
\begin{eqnarray}
g_{\nu }(z) &=&\frac{2^{\nu }\Gamma (\nu +1)z^{1-\nu }}{Q(\nu)}N_{\nu }(z)
\label{S1} \\
&=&z\prod\limits_{n\geq 1}\left( 1-\frac{z^{2}}{ \lambda_{\nu ,n}^{2}}%
\right) .  \notag
\end{eqnarray}%
On the other hand, the series representation of $g_{\nu }(z)$%
\begin{equation}
g_{\nu }(z)=\sum_{n=0}^{\infty }\frac{\left( -1\right) ^{n}\Gamma (\nu
+1)Q(2n+\nu)}{n!4^{n}\Gamma (n+\nu +1)Q(\nu)}z^{2n+1}.  \label{S2}
\end{equation}

Now, we would like to mention that by using the equations (\ref{S1}) and (%
\ref{S2}) we can obtain the following Euler-Rayleigh sums for the positive
zeros of the function $g_{\nu}$. From the equality (\ref{S2}) we have%
\begin{equation}
g_{\nu }(z)=z-\frac{Q(\nu +2)}{4(\nu+1)Q(\nu) }z^{3}+\frac{Q(\nu +4)}{%
32(\nu+1)(\nu +2)Q(\nu)}z^{5}-\cdots .  \label{S3}
\end{equation}%
Now, if we consider (\ref{S1}), then some calculations yield that%
\begin{equation}
g_{\nu }(z)=z-\sum\limits_{n\geq 1}\frac{1}{ \lambda_{\nu ,n} ^{2}}z^{3}+%
\frac{1}{2}\left( \left( \sum\limits_{n\geq 1}\frac{1}{\lambda_{\nu ,n}^{2}}%
\right) ^{2}-\sum\limits_{n\geq 1}\frac{1}{ \lambda_{\nu ,n}^{4}}\right)
z^{5}-\cdots .  \label{S4}
\end{equation}%
By equating the first few coefficients with the same degrees in equations (%
\ref{S3}) and (\ref{S4}) we get,%
\begin{equation}
\sum\limits_{n\geq 1}\frac{1}{ \lambda_{\nu ,n}^{2}}= \frac{Q(\nu +2)}{%
4(\nu+1)Q(\nu) }  \label{S5}
\end{equation}%
and%
\begin{equation}
\sum\limits_{n\geq 1}\frac{1}{\lambda_{\nu ,n}^{4}}=\frac{1}{16(\nu+1)Q(\nu)}%
\left( \frac{Q^2(\nu +2)}{(\nu +1)Q(\nu)}- \frac{Q(\nu +4)}{(\nu+2)}\right) .
\label{S6}
\end{equation}

\section{Main \ Results}

\subsection{Radii of Starlikeness and Convexity of The Functions $f_{ 
\protect\nu },$ $g_{\protect\nu }$ and $h_{\protect\nu }$}

The first principal result we established concerns the radii of starlikeness
and reads as follows.

\begin{theorem}
\label{t2} Let $\beta \in
\lbrack 0,1).$ The following statements hold:

\begin{description}
\item[a)] If $\nu \geq \max\{0,\nu_{0} \}$, $\nu\neq0$  then the radius of starlikeness of order $\beta$ of the
function $f_\nu$ is the smallest positive root of the equation 
\begin{equation*}
ar^3J_\nu^{\prime\prime\prime}(r)+\left(2a+b-a\nu\beta
\right)r^2J_\nu^{\prime\prime}(r)+\left(b+c-b\nu\beta
\right)rJ_\nu^{\prime}(r)-c\nu \beta J_\nu(r)=0.
\end{equation*}

\item[b)] If $\nu \geq \max\{0,\nu_{0} \},$ then the radius of starlikeness of order $\beta$ of the function $%
g_\nu$ is the smallest positive root of the equation 
\begin{equation*}
\left(1-\nu \right)+\frac{ar^3J_\nu^{\prime\prime\prime}(r)+\left(2a+b%
\right)r^2J_\nu^{\prime\prime}(r)+\left(b+c\right)rJ_\nu^{\prime}(r)}{%
ar^2J_\nu^{\prime\prime}(r)+brJ_\nu^{\prime}(r)+cJ_\nu(r)}=\beta
\end{equation*}

\item[c)] If $\nu \geq \max\{0,\nu_{0} \},$ then the radius of starlikeness of order $\beta$ of the function $%
h_\nu$ is the smallest positive root of the equation 
\begin{equation*}
\left(2-\nu \right)+\frac{ar\sqrt{r} J_\nu^{\prime\prime\prime}(\sqrt{r}%
)+\left(2a+b\right)rJ_\nu^{\prime\prime}(\sqrt{r})+\left(b+c\right)\sqrt{r}%
J_\nu^{\prime}(\sqrt{r})}{arJ_\nu^{\prime\prime}(\sqrt{r})+b\sqrt{r}%
J_\nu^{\prime}(\sqrt{r})+cJ_\nu(\sqrt{r})}=2\beta
\end{equation*}
\end{description}
\end{theorem}

\begin{proof}
Firstly, we prove part \textbf{a} for $\nu \geq \max \{0,\nu_0\},~\nu\neq0 $
and \textbf{b} and \textbf{c} for $\nu \geq \max \{0,\nu_0\}.$ We need to
show that the following inequalities 
\begin{equation}
\operatorname{Re}\left( \frac{zf_{\nu }^{\prime }(z)}{f_{\nu }(z)}\right) >\beta,
~~~ \operatorname{Re}\left( \frac{zg_{\nu }^{\prime }(z)}{g_{\nu }(z)}\right) >\beta
~~~ \text{ and }~~~ \operatorname{Re}\left( \frac{zh_{\nu }^{\prime }(z)}{h_{\nu }(z)%
}\right) >\beta  \label{re1}
\end{equation}%
are valid for $z\in \mathbb{D}_{r_{\beta }^{\ast }(f_{\nu })},~z\in \mathbb{D%
}_{r_{\beta }^{\ast }(g_{\nu })}$ and $z\in \mathbb{D}_{r_{\beta }^{\ast
}(h_{\nu })}$ respectively, and each of the above inequalities does not hold
in larger disks.

When we write the equation (\ref{r11}) in definition of the functions $%
f_{\nu }(z),~g_{\nu }(z)$ and $h_{\nu }(z)$ we get by using logarithmic
derivation 
\begin{eqnarray*}
\frac{zf_{\nu }^{\prime }(z)}{f_{\nu }(z)} &=&\frac{1}{\nu }\frac{ zN_{\nu
}^\prime(z)}{N_{\nu }(z)}=1-\frac{1}{\nu}\sum\limits_{n\geq 1}\frac{2z^{2}}{
\lambda_{\nu ,n} ^{2}-z^{2} },\text{ \ \ }\left(\nu \geq \max
\{0,\nu_0\},~\nu\neq0\right), \\
\frac{zg_{\nu }^{\prime }(z)}{g_{\nu }(z)} &=&(1-\nu) +\frac{zN_{\nu
}^{\prime}(z)}{N_{\nu }(z)}=1- \sum\limits_{n\geq 1}\frac{2z^{2}}{
\lambda_{\nu ,n}^2-z^{2}},\text{ \ \ } \left(\nu \geq \max\{0,\nu_0
\}\right), \\
\frac{zh_{\nu }^{\prime }(z)}{h_{\nu }(z)} &=&(1-\frac{\nu}{2}) +\frac{1}{2} 
\frac{\sqrt{z} N_{\nu }^{\prime}(\sqrt{z})}{N_{\nu }(\sqrt{z})}=1-
\sum\limits_{n\geq 1}\frac{z}{ \lambda_{\nu ,n}^2-z},\text{ \ \ }(\nu \geq
\max\{0,\nu_0 \}).
\end{eqnarray*}%
It is known \cite{Ba1} that if $z\in 
%TCIMACRO{\U{2102} }%
%BeginExpansion
\mathbb{C}
%EndExpansion
$ and $\lambda \in 
%TCIMACRO{\U{211d} }%
%BeginExpansion
\mathbb{R}
%EndExpansion
$ are such that $\lambda >\left\vert z\right\vert ,$ then%
\begin{equation}
\frac{\left\vert z\right\vert }{\lambda -\left\vert z\right\vert }\geq \operatorname{%
Re}\left( \frac{z}{\lambda -z}\right) .  \label{r2}
\end{equation}%
Then the inequality%
\begin{equation*}
\frac{\left\vert z\right\vert ^{2}}{ \lambda_{\nu ,n}^ {2}-\left\vert
z\right\vert ^{2}}\geq \operatorname{Re}\left( \frac{z^{2}}{ \lambda_{\nu
,n}^{2}-z^{2}}\right)
\end{equation*}%
holds for every $\nu\geq \max\{0,\nu_0 \}, \nu \neq 0$ and $z\in \mathbb{D}%
(0,\lambda_{\nu ,1})$ Therefore,%
\begin{eqnarray*}
\operatorname{Re}\left( \frac{zf_{\nu }^{\prime }(z)}{f_{\nu }(z)}\right) &=&1-\frac{%
1}{\nu}\sum\limits_{n\geq 1}\operatorname{Re}\left( \frac{2z^{2}}{ \lambda_{\nu
,n}^{2}-z^{2}}\right) \geq 1-\frac{1}{\nu }\sum\limits_{n\geq 1}\frac{%
2\left\vert z\right\vert ^{2}}{ \lambda_{\nu ,n}^{2}-\left\vert z\right\vert
^{2}}=\frac{\left\vert z\right\vert f_{\nu }^{\prime }(\left\vert
z\right\vert )}{f_{\nu }(\left\vert z\right\vert )}, \\
\operatorname{Re}\left( \frac{zg_{\nu }^{\prime }(z)}{g_{\nu }(z)}\right)
&=&1-\sum\limits_{n\geq 1}\operatorname{Re}\left( \frac{2z^{2}}{ \lambda_{\nu
,n}^{2}-z^{2}}\right) \geq 1-\sum\limits_{n\geq 1}\frac{2\left\vert
z\right\vert ^{2}}{ \lambda_{\nu ,n} ^{2}-\left\vert z\right\vert ^{2}}=%
\frac{\left\vert z\right\vert g_{\nu }^{\prime }(\left\vert z\right\vert )}{%
g_{\nu }(\left\vert z\right\vert ) }, \\
\operatorname{Re}\left( \frac{zh_{\nu }^{\prime }(z)}{h_{\nu }(z)}\right)
&=&1-\sum\limits_{n\geq 1}\operatorname{Re}\left( \frac{z}{ \lambda_{\nu ,n}^{2}-z}%
\right) \geq 1-\sum\limits_{n\geq 1}\frac{\left\vert z\right\vert}{
\lambda_{\nu ,n} ^{2}-\left\vert z\right\vert }=\frac{\left\vert
z\right\vert h_{\nu }^{\prime }(\left\vert z\right\vert )}{h_{\nu
}(\left\vert z\right\vert )},
\end{eqnarray*}%
where equalities are attained only when $z=\left\vert z\right\vert =r.$
Thus, for $r\in \left(0,\lambda_{\nu ,1} \right)$ it follows that 
\begin{equation*}
\inf_{z\in \mathbb{D}_r } \left\lbrace \operatorname{Re} \left( \frac{zf_{\nu
}^{\prime }(z)}{f_{\nu }(z)} \right) \right\rbrace= \frac{rf_{\nu }^{\prime
}(r)}{f_{\nu }(r)},~~~ \inf_{z\in \mathbb{D}_r } \left\lbrace \operatorname{Re}
\left( \frac{zg_{\nu }^{\prime }(z)}{g_{\nu }(z)} \right) \right\rbrace= 
\frac{rg_{\nu }^{\prime }(r)}{g_{\nu }(r)}
\end{equation*}
and 
\begin{equation*}
\inf_{z\in \mathbb{D}_r } \left\lbrace \operatorname{Re} \left( \frac{zh_{\nu
}^{\prime }(z)}{h_{\nu }(z)} \right) \right\rbrace= \frac{rh_{\nu }^{\prime
}(r)}{h_{\nu }(r)}. 
\end{equation*}
On the other and, the mappings $\psi_\nu,~\varphi_\nu,~\phi_\nu:\left(0,%
\lambda_{\nu ,1} \right)\longrightarrow \mathbb{R}$ defined by 
\begin{equation*}
\psi_\nu(r)= \frac{rf_{\nu }^{\prime }(r)}{f_{\nu }(r)}=1-\frac{1}{\nu}%
\sum\limits_{n\geq 1}\left( \frac{2r^{2}}{ \lambda_{\nu ,n}^{2}-r^{2}}%
\right), \text{ \ \ \ } \varphi_\nu(r)= \frac{rg_{\nu }^{\prime }(r)}{g_{\nu
}(r)}=1- \sum\limits_{n\geq 1}\left( \frac{2r^{2}}{ \lambda_{\nu
,n}^{2}-r^{2}}\right) 
\end{equation*}
and 
\begin{equation*}
\phi_\nu(r)= \frac{rh_{\nu }^{\prime }(r)}{h_{\nu }(r)}=1-
\sum\limits_{n\geq 1}\left( \frac{r}{ \lambda_{\nu ,n}^{2}-r}\right). 
\end{equation*}
Since, 
\begin{equation*}
\psi_\nu^\prime(r)= -\frac{1}{\nu}\sum\limits_{n\geq 1}\left( \frac{%
4r\lambda_{\nu,n}^{2}}{ \left(\lambda_{\nu ,n}^{2}-r^{2}\right)^2}\right)<0,%
\text{ \ \ \ } \varphi_\nu^\prime(r)= -\sum\limits_{n\geq 1}\left( \frac{%
4r\lambda_{\nu,n}^{2}}{ \left(\lambda_{\nu ,n}^{2}-r^{2}\right)^2}\right)<0 
\end{equation*}
and 
\begin{equation*}
\phi_\nu^\prime(r)= -\sum\limits_{n\geq 1}\left( \frac{\lambda_{\nu,n}^{2}}{
\left(\lambda_{\nu ,n}^{2}-r\right)^2}\right)<0 
\end{equation*}
$\psi_\nu,~\varphi_\nu,~\phi_\nu$ are strictly decreasing for all $\nu\geq
\max\{0,\nu_0 \}, \nu \neq 0.$ Now, since $\lim_{r\searrow0}\psi_\nu(r)=1>%
\beta,$ $\lim_{r\searrow0}\varphi_\nu(r)=1>\beta,$ $\lim_{r\searrow0}\phi_%
\nu(r)=1>\beta$ and $\lim_{r\nearrow \lambda_{\nu ,1}}\psi_\nu(r)=-\infty,$ $%
\lim_{r\nearrow \lambda_{\nu ,1}}\varphi_\nu(r)=-\infty,$ $\lim_{r\nearrow
\lambda_{\nu ,1}}\phi_\nu(r)=-\infty,$ in view of the minimum principle for
harmonic functions imply that the corresponding inequalities in (\ref{re1})
for $\nu\geq \max\{0,\nu_0 \}, \nu \neq 0$ hold if only if $z\in \mathbb{D}%
_{r_1},$ $z\in \mathbb{D}_{r_2}$ and $z\in \mathbb{D}_{r_3},$ respectively,
where $r_1,$ $r_2$ and $r_3$ is the smallest positive roots of euations 
\begin{equation*}
\frac{rf_{\nu }^{\prime }(r)}{f_{\nu }(r)}=\beta,~~ \frac{rg_{\nu }^{\prime
}(r)}{g_{\nu }(r)}=\beta \text{ and } \frac{rh_{\nu }^{\prime }(r)}{h_{\nu
}(r)}=\beta 
\end{equation*}
which are equivalent to 
\begin{equation*}
\frac{rN_\nu^\prime(r)}{\nu N_\nu(r)}=\beta,~~ \left(1-\nu \right)+\frac{%
rN_\nu^\prime(r)}{ N_\nu(r)}=\beta \text{ and } \left(1-\frac{\nu}{2}
\right)+\frac{1}{2} \frac{\sqrt{r} N_{\nu }^{\prime}(\sqrt{r})}{N_{\nu }(%
\sqrt{r})}=\beta 
\end{equation*}
situated in $\left(0,\lambda_{\nu ,1}\right).$\newline
This completes the proof of part \textbf{a} when $\nu\geq \max\{0,\nu_0 \},
\nu \neq 0$, and parts \textbf{b} and \textbf{c} when $\nu\geq \max\{0,\nu_0
\}.$
\end{proof} 
\begin{eqnarray*}
&&
\begin{tabular}{c|c|c||c|c|c||c|c|c|}
\cline{2-9}
& \multicolumn{8}{|c|}{$r^{\ast }\left( f\right) $} \\ \cline{2-9}
& \multicolumn{2}{|c||}{$b=1$ and $c=0$} &  & \multicolumn{2}{|c||}{$a=1$
and $c=0$} &  & \multicolumn{2}{|c|}{$a=1$ and $b=2$} \\ 
\cline{2-3}\cline{5-6}\cline{8-9}
& $\beta =0$ & $\beta =0.5$ &  & $\beta =0$ & $\beta =0.5$ &  & $\beta =0$ & 
$\beta =0.5$ \\ \hline
\multicolumn{1}{|c|}{$a=2$} & $0.8231$ & $0.6458$ & $b=2$ & $1.0917$ & $%
0.8481$ & $c=2$ & $1.3089$ & $1.0058$ \\ \hline
\multicolumn{1}{|c|}{$a=3$} & $0.7689$ & $0.6045$ & $b=3$ & $1.1774$ & $%
0.9120$ & $c=3$ & $1.3952$ & $1.0671$ \\ \hline
\multicolumn{1}{|c|}{$a=4$} & $0.7382$ & $0.5809$ & $b=4$ & $1.2337$ & $%
0.9539$ & $c=4$ & $1.4708$ & $1.1203$ \\ \hline
\end{tabular}
\\
&& \\
&&\text{\ \ \ \ \ \ \ \ \ \ \ \ \ \ \ \ \textbf{Table1.} Radii of
starlikeness for }f_{\nu }\text{ when }\nu =1.5
\end{eqnarray*}%
\begin{eqnarray*}
&&%
\begin{tabular}{c|c|c||c|c|c||c|c|c|}
\cline{2-9}
& \multicolumn{8}{|c|}{$r^{\ast }\left( g\right) $} \\ \cline{2-9}
& \multicolumn{2}{|c||}{$b=1$ and $c=0$} &  & \multicolumn{2}{|c||}{$a=1$
and $c=0$} &  & \multicolumn{2}{|c|}{$a=1$ and $b=2$} \\ 
\cline{2-3}\cline{5-6}\cline{8-9}
& $\beta =0$ & $\beta =0.5$ &  & $\beta =0$ & $\beta =0.5$ &  & $\beta =0$ & 
$\beta =0.5$ \\ \hline
\multicolumn{1}{|c|}{$a=2$} & $0.7188$ & $0.5483$ & $b=2$ & $0.9477$ & $%
0.7167$ & $c=2$ & $1.1285$ & $0.8459$ \\ \hline
\multicolumn{1}{|c|}{$a=3$} & $0.6723$ & $0.5137$ & $b=3$ & $1.0203$ & $%
0.7697$ & $c=3$ & $1.1995$ & $0.8957$ \\ \hline
\multicolumn{1}{|c|}{$a=4$} & $0.6458$ & $0.4939$ & $b=4$ & $1.0678$ & $%
0.8044$ & $c=4$ & $1.2611$ & $0.9388$ \\ \hline
\end{tabular}
\\
&& \\
&&\text{\ \ \ \ \ \ \ \ \ \ \ \ \ \ \ \ \textbf{Table2.} Radii of
starlikeness for }g_{\nu }\text{ when }\nu =1.5
\end{eqnarray*}%
\begin{eqnarray*}
&&%
\begin{tabular}{c|c|c||c|c|c||c|c|c|}
\cline{2-9}
& \multicolumn{8}{|c|}{$r^{\ast }\left( h\right) $} \\ \cline{2-9}
& \multicolumn{2}{|c||}{$b=1$ and $c=0$} &  & \multicolumn{2}{|c||}{$a=1$
and $c=0$} &  & \multicolumn{2}{|c|}{$a=1$ and $b=2$} \\ 
\cline{2-3}\cline{5-6}\cline{8-9}
& $\beta =0$ & $\beta =0.5$ &  & $\beta =0$ & $\beta =0.5$ &  & $\beta =0$ & 
$\beta =0.5$ \\ \hline
\multicolumn{1}{|c|}{$a=2$} & $0.8009$ & $0.5167$ & $b=2$ & $1.4211$ & $%
0.8982$ & $c=2$ & $2.0638$ & $1.2737$ \\ \hline
\multicolumn{1}{|c|}{$a=3$} & $0.6979$ & $0.4520$ & $b=3$ & $1.6575$ & $%
1.0410$ & $c=3$ & $2.3560$ & $1.4388$ \\ \hline
\multicolumn{1}{|c|}{$a=4$} & $0.6426$ & $0.4171$ & $b=4$ & $1.8225$ & $%
1.1403$ & $c=4$ & $2.6296$ & $1.5905$ \\ \hline
\end{tabular}
\\
&& \\
&&\text{\ \ \ \ \ \ \ \ \ \ \ \ \ \ \ \ \textbf{Table3.} Radii of
starlikeness for }h_{\nu }\text{ when }\nu =1.5
\end{eqnarray*} 
When $\nu=1.5,$ considering the special values of $a,b,c\in \mathbb{R},$ radii of starlikeness of the functions $f_\nu,g_\nu$ and $h_\nu$ is seen from the tables above. If the values of $b$ and $c$ are kept constant and the value of $a$ is increased, radii of starlikeness of the functions $f_\nu,g_\nu$ and $h_\nu$ is monotone decreasing. If the values of $a$ and $c$ are kept constant and the value of $b$ is increased or the values of $a$ and $b$ are kept constant and the value of $c$ is increased radii of starlikeness of the functions $f_\nu,g_\nu$ and $h_\nu$ is monotone increasing. In addition, according to the increasing values of $\beta,$ it is clear that radii of starlikeness of the functions $f_\nu,g_\nu$ and $h_\nu$ is decreasing.
\begin{center}
\includegraphics{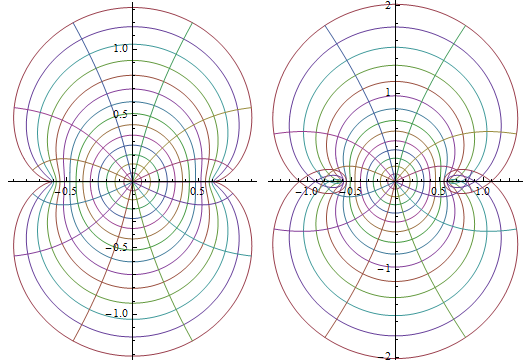}

\end{center}
\begin{center}
\text{\textbf{Figure1.}} Images of function $g_{3/2}(z)$ for $r^\ast(g_{3/2})=0.9477$ and $r^\ast(g_{3/2})=1.2$, respectively
\end{center}

The second principal result we established concerns the radii of convexity
and reads as follows.

\begin{theorem}
\label{t3} Let $\beta \in \lbrack 0,1).$ The following statements hold:
\begin{description}
\item[a)] If $\nu\geq\max\left\{0,\nu_0 \right\},\ \nu \neq 0$
then, the radius $r_{\beta }^{c}(f_{\nu })$ is the smallest positive root of
the equation
\begin{eqnarray*}
& & 1+\frac{rN_{\nu}^{\prime \prime }(r)}{N_{\nu }^{\prime }(r)}+\left(\frac{%
1}{\nu}-1\right)\frac{rN_{\nu}^{ \prime }(r)}{N_{\nu }(r)}=\beta
\end{eqnarray*}
Moreover, $r_{\beta }^{c}(f_{\nu})<\lambda_{\nu ,1}^{\prime}<\lambda_{\nu
,1}.$

\item[b)] If $\nu\geq\max\left\{0,\nu_0 \right\}$
then, the radius $r_{\beta }^{c}(g_{\nu })$ is the smallest positive root of
the equation 
\begin{equation*}
1+\frac{r^2N_{\nu}^{\prime \prime }(r)+\left(2-2\nu\right)rN_{\nu}^{\prime
}(r)+\left(\nu^2-\nu\right)N_{\nu}(r)}{rN_{\nu}^{\prime
}(r)+\left(1-\nu\right)N_{\nu}(r)}=\beta.
\end{equation*}

\item[c)] If $\nu\geq\max\left\{0,\nu_0 \right\}$
then, the radius $r_{\beta }^{c}(h_{\nu })$ is the smallest positive root of
the equation 
\begin{equation*}
1+\frac{rN_{\nu}^{\prime \prime }(\sqrt{r})+\left(3-2\nu\right)\sqrt{r}%
N_{\nu}^{\prime }(\sqrt{r})+\left(\nu^2-\nu\right)N_{\nu}(\sqrt{r})}{2\sqrt{r%
}N_{\nu}^{\prime }(\sqrt{r})+2\left(2-\nu\right)N_{\nu}(\sqrt{r})}=\beta.
\end{equation*}
\end{description}
\end{theorem}

\begin{proof}
\textbf{a)} Since 
\begin{equation*}
1+\frac{zf_{\nu }^{\prime \prime }(z)}{f_{\nu }^{\prime }(z)}=1+\frac{
zN_{\nu }^{\prime\prime}(z)}{N_{\nu }^{\prime}(z)}+\left( \frac{1}{\nu }%
-1\right) \frac{zN_{\nu }^{\prime}(z)}{N_{\nu }(z)}
\end{equation*}%
and by means of (\ref{r11}) we have%
\begin{equation*}
\frac{zN_{\nu }^{\prime}(z)}{N_{\nu }(z)}=\nu -\sum\limits_{n\geq 1}\frac{%
2z^{2}}{ \lambda_{\nu,n } ^{2}-z^{2}}.
\end{equation*}
Moreover, we obtain 
\begin{equation*}
N_{\nu }^\prime(z)=\sum_{n=0}^{\infty }\frac{\left(2n+\nu\right) Q(2n+\nu
)\left( -1\right) ^{n}}{2.n!\Gamma (n+\nu +1)}\left( \frac{z}{2}\right)
^{2n+\nu }\text{ \ \ }\left( z\in \mathbb{C} \right)
\end{equation*}
and 
\begin{equation*}
\frac{2^\nu \Gamma\left(\nu+1\right)}{ Q(\nu)z^{\nu-1}\nu }N_{\nu
}^\prime(z)=1+\sum_{n=1}^{\infty }\frac{\left(2n+\nu\right) Q(2n+\nu )\left(
-1\right) ^{n}}{2.n!\Gamma (n+\nu +1)}\left( \frac{z}{2}\right) ^{2n+\nu }%
\text{ \ \ }\left( z\in \mathbb{C} \right).
\end{equation*}
Taking into consideration the well-known limit 
\begin{equation*}
\lim\limits_{n\longrightarrow \infty} \frac{\log\Gamma\left(n+c\right)}{n
\log n}=1, 
\end{equation*}
where $c$ is a positive constant, and Levin \cite{Le} we infer that above entire
function is of growth order $\rho=\frac{1}{2}.$ Namely, for $\nu>
\max\left\{0,\nu_0 \right\} $ we have that as  
\begin{equation*}
\lim\limits_{n\longrightarrow \infty}\frac{n \log n}{n\log4+\log\left[%
\Gamma\left(n+\nu+1 \right) \right]+\log\left[\Gamma\left(n+1 \right) \right]%
-\log\left[Q\left(2n+\nu\right) \right]-\log\left(2n+\nu\right)} =\frac{1}{2}
. 
\end{equation*}
Now, by applying Hadamard's Theorem \cite{Le} we can write the infinite product
representation of $N_\nu^\prime(z)$ as follows:

\begin{equation*}
\frac{Q(\nu ) z^{\nu-1 } \nu}{2^{\nu }\Gamma (\nu +1)}\prod\limits_{n\geq
1}\left( 1-\frac{z^{2}}{ \lambda_{\nu ,n} ^{\prime2}}\right)
\end{equation*}
where $\lambda_{\nu ,n} ^{\prime}$ denotes the $n$th positive zero of the
function $N_\nu^\prime(z).$

Observe also that 
\begin{equation*}
1+\frac{zf_{\nu }^{\prime \prime }(z)}{f_{\nu }^{\prime }(z)}=1-\left( \frac{%
1}{\nu }-1\right) \sum\limits_{n\geq 1}\frac{2z^{2}}{ \lambda_{\nu ,n}
^{2}-z^{2}}-\sum\limits_{n\geq 1}\frac{2z^{2}}{\lambda_{\nu ,n}
^{\prime2}-z^{2}}.
\end{equation*}%
Now, by using the inequality (\ref{r2}), for all $z\in \mathbb{D}%
_{\lambda_{\nu ,1}^{\prime}}$ and $1\geq \nu>\max\left\{0,\nu_0 \right\}$ we
obtain the inequality 
\begin{eqnarray*}
\operatorname{Re}\left( 1+\frac{zf_{\nu }^{\prime \prime }(z)}{f_{\nu }^{\prime }(z)}%
\right) &=&1-\left( \frac{1}{\nu }-1\right) \sum\limits_{n\geq 1}\operatorname{Re}%
\left( \frac{2z^{2}}{ \lambda_{\nu ,n} ^{2}-z^{2}}\right)
-\sum\limits_{n\geq 1}\operatorname{Re}\left( \frac{2z^{2}}{\lambda_{\nu ,n}
^{\prime2}-z^{2}}\right) \\
&\geq &1-\left( \frac{1}{\nu }-1\right) \sum\limits_{n\geq 1}\frac{2r^{2} }{
\lambda_{\nu,n } ^{2}-r^{2}}-\sum\limits_{n\geq 1}\frac{2r^{2}}{\lambda_{\nu
,n} ^{\prime2}-r^{2}}
\end{eqnarray*}
where $\left\vert z\right\vert =r.$ Moreover, observe that if we use the
inequality \cite[Lemma 2.1]{Ba1}%
\begin{equation*}
\mu \operatorname{Re}\left( \frac{z}{a-z}\right) -\operatorname{Re}\left( \frac{z}{b-z}%
\right) \geq \mu \frac{\left\vert z\right\vert }{a-\left\vert z\right\vert }-%
\frac{\left\vert z\right\vert }{b-\left\vert z\right\vert }
\end{equation*}%
where $a>b>0,$ $\mu \in \lbrack 0,1]$ and $z\in 
%TCIMACRO{\U{2102} }%
%BeginExpansion
\mathbb{C}
%EndExpansion
$ such that $|z|<b$, then we get that the above inequality is also valid
when $\nu\geq 1$. Here we used that the zeros of $N_{\nu}$ and  $%
N_{\nu}^\prime$ are interlacing according to Lemma \ref{Le4}. The above
inequality implies for $r\in (0,\lambda_{\nu ,1}^{\prime})$%
\begin{equation*}
\underset{z\in \mathbb{D}_{r}}{\inf }\left\{ \operatorname{Re}\left( 1+\frac{%
zf_{\nu}^{\prime \prime }(z)}{f_{\nu }^{\prime }(z)}\right) \right\} =1+%
\frac{rf_{\nu }^{\prime \prime }(r)}{f_{\nu }^{\prime }(r)}.
\end{equation*}%
On the other hand, we define the function $\Lambda _{\nu }:(0,\lambda_{\nu
,1}^{\prime})\rightarrow \mathbb{R} ,$ 
\begin{equation*}
\Lambda _{\nu }(r)=1+\frac{rf_{\nu }^{\prime \prime }(r)}{f_{\nu }^{\prime
}(r)}.
\end{equation*}

Since the zeros of $N_{\nu}$ and derivative of $N_{\nu}$ are interlacing
according to Lemma \ref{Le4} and $r<\lambda_{\nu ,1}^{\prime}<\lambda_{\nu ,1}
$ $\left( \text{or }r<\sqrt{\lambda_{\nu ,1}\lambda_{\nu ,1}^{\prime}}%
\right) $ for all $\nu>\max\left\{0,\nu_0 \right\}$ we have 
\begin{equation*}
\left( \lambda_{\nu ,n}\right) \left( \lambda_{\nu ,n}^{\prime
2}-r^{2}\right) -\left( \lambda_{\nu ,n}^{\prime}\right) \left( \lambda_{\nu
,n}^{2}-r^{2}\right) <0.
\end{equation*}%
Thus following inequality 
\begin{eqnarray*}
\Lambda _{\nu }^\prime(r) &=&-\left( \frac{1}{\nu }-1\right)
\sum\limits_{n\geq 1}\frac{4r\lambda_{\nu ,n} ^{2}}{ \left(\lambda_{\nu ,n}
^{2}-r^{2}\right)^2}-\sum\limits_{n\geq 1}\frac{4r\lambda_{\nu ,n} ^{\prime2}%
}{\left(\lambda_{\nu ,n} ^{\prime2}-r^{2}\right)^2} \\
&<& \sum\limits_{n\geq 1}\frac{4r\lambda_{\nu ,n} ^{2}}{ \left(\lambda_{\nu
,n} ^{2}-r^{2}\right)^2}-\sum\limits_{n\geq 1}\frac{4r\lambda_{\nu ,n}
^{\prime2}}{\left(\lambda_{\nu ,n} ^{\prime2}-r^{2}\right)^2} \\
&=&4r\sum\limits_{n\geq 1}\frac{\left( \lambda_{\nu ,n}\right)^2 \left(
\lambda_{\nu ,n}^{\prime 2}-r^{2}\right)^2 -\left( \lambda_{\nu
,n}^{\prime}\right)^2 \left( \lambda_{\nu ,n}^{2}-r^{2}\right)^2 }{\left(
\lambda_{\nu ,n}^{2}-r^{2}\right)^2 \left( \lambda_{\nu ,n}^{\prime
2}-r^{2}\right)^2 }<0
\end{eqnarray*}%
is satisfied. Consequently, the function $\Lambda _{\nu }$ is strictly
decreasing. Observe also that $\lim_{r\searrow 0}\Lambda _{\nu }(r)=1>\beta $
and $\lim_{r\nearrow \lambda_{\nu ,1}^{\prime}}\Lambda _{\nu }(r)=-\infty ,$
which means that for $z\in \mathbb{D}_{r_{4}}$ we have%
\begin{equation*}
\operatorname{Re}\left( 1+\frac{zf_{\nu }^{\prime \prime }(z)}{f_{\nu }^{\prime }(z)}%
\right) >\beta
\end{equation*}%
if and ony if $r_{4}$ is the unique root of 
\begin{equation*}
1+\frac{rf_{\nu}^{\prime \prime }(r)}{f_{\nu }^{\prime }(r)}=\beta \text{ or 
} 1+\frac{rN_{\nu}^{\prime \prime }(r)}{N_{\nu }^{\prime }(r)}+\left(\frac{1%
}{\nu}-1\right)\frac{rN_{\nu}^{ \prime }(r)}{N_{\nu }(r)}=\beta,
\end{equation*}%
situated in $(0,\lambda_{\nu ,1}^{\prime}).$

\textbf{b)} Observe that 
\begin{equation*}
1+\frac{zg_{\nu }^{\prime \prime }(z)}{g_{\nu}^{\prime }(z)}=1+ \frac{%
z^2N_{\nu }^{\prime \prime}(z)+2(1-\nu)zN_{\nu }^{\prime}(z)+(\nu^2
-\nu)N_{\nu }(z)}{zN_{\nu }^{\prime}(z)+(1-\nu)N_{\nu }(z)}.
\end{equation*}
By using (\ref{F1}) and (\ref{r11}) we have 
\begin{eqnarray}
g_{\nu }^{\prime }(z) &=&\frac{2^{\nu }\Gamma (\nu +1)z^{-\nu }}{Q(\nu)}%
\left[ (1-\nu )N_{\nu }(z)+zN_{\nu }^{\prime}(z)\right]  \label{r31} \\
&=&\sum_{n=0}^{\infty }\frac{\left( -1\right) ^{n}(2n+1)Q(2n+\nu)\Gamma (\nu
+1)}{n!\Gamma (n+\nu +1)Q(\nu)}\left( \frac{z}{2}\right) ^{2n}  \notag
\end{eqnarray}%
and 
\begin{equation*}
\lim\limits_{n\longrightarrow \infty}\frac{n \log n}{%
\begin{array}{c}
\lbrack n\log4+\log\left[\Gamma\left(n+\nu+1 \right) \right]+\log\left[%
\Gamma\left(n+1 \right) \right]+\log\left[Q\left(\nu\right)\right] \\ 
-\log\left[Q\left(2n+\nu\right) \right]-\log\left(2n+1\right)-\log\left[%
\Gamma\left(\nu+1 \right) \right]%
\end{array}%
} = \frac{1}{2} .
\end{equation*}%
Here, we used $n!=\Gamma (n+1)$ and $\underset{n\rightarrow \infty }{\lim }%
\frac{\log \Gamma (an+b)}{n\log n}=a,$ where $b$ and $c$ are positive
constants. So, by applying Hadamard's Theorem \cite[p. 26]{Le} we can write
the infinite product representation of $g_{\nu }^{\prime }(z)$ as follows: 
\begin{equation}
g_{\nu }^{\prime }(z)=\prod\limits_{n\geq 1}\left( 1-\frac{z^{2}}{ \delta
_{\nu ,n}^{2}}\right) ,  \label{r4}
\end{equation}%
where $\delta _{\nu ,n}$ denotes the $n$th positive zero of the function $%
g_{\nu }^{\prime }.$ From Lemma \ref{Le4} for $\nu>\max\left\{0,\nu_0
\right\}$ the function $g_{\nu }^{\prime }\in \mathcal{LP}$, and the
smallest positive zero of $g_{\nu }^{\prime }$ does not exceed the first
positive zero of $N_{\nu }.$

By means of (\ref{r4}) we have 
\begin{equation*}
1+\frac{zg_{\nu }^{\prime \prime }(z)}{g_{\nu }^{\prime }(z)} =
1-\sum\limits_{n\geq 1}\operatorname{Re}\left( \frac{2z^{2}}{\delta_{\nu ,n}
^{2}-z^{2}}\right).
\end{equation*}
By using the inequality (\ref{r2}), for all $z\in \mathbb{D}_{\delta _{\nu
,n}}$ we obtain the inequality 
\begin{equation*}
\operatorname{Re}\left( 1+\frac{zg_{\nu }^{\prime \prime }(z)}{g_{\nu }^{\prime }(z)}%
\right) \geq 1 -\sum\limits_{n\geq 1}\frac{2r^{2}}{\delta_{\nu ,n} ^{2}-r^{2}%
}
\end{equation*}%
where $\left\vert z\right\vert =r.$ Thus, for $r\in (0,\delta _{\nu ,1})$ we
get 
\begin{equation*}
\underset{z\in \mathbb{D}_{r}}{\inf }\left\{ \operatorname{Re}\left( 1+\frac{zg_{\nu
}^{\prime \prime }(z)}{g_{\nu }^{\prime }(z)}\right) \right\} =1+\frac{%
rg_{\nu }^{\prime \prime }(r)}{g_{\nu }^{\prime }(r)}.
\end{equation*}%
The function $\varTheta_{\nu}:(0,\delta _{\nu ,1})\rightarrow 
%TCIMACRO{\U{211d} }%
%BeginExpansion
\mathbb{R}
%EndExpansion
,$ defined by%
\begin{equation*}
\varTheta_{\nu }(r)=1+\frac{rg_{\nu }^{\prime \prime }(r)}{g_{\nu }^{\prime
}(r)},
\end{equation*}%
is strictly decreasing and $\lim_{r\searrow 0}\varTheta_{\nu }(r)=1>\beta $
and $\lim_{r\nearrow \delta _{\nu ,1}}\varTheta_{\nu ,n}(r)=-\infty $.
Consequently, in view of the minimum principle for harmonic functions for $%
z\in \mathbb{D}_{r_{5}}$ we have that%
\begin{equation*}
\operatorname{Re}\left( 1+\frac{zg_{\nu }^{\prime \prime }(z)}{g_{\nu }^{\prime }(z)}%
\right) >\beta
\end{equation*}%
if and ony if $r_{5}$ is the unique root of 
\begin{equation*}
1+\frac{rg_{\nu }^{\prime \prime }(r)}{g_{\nu }^{\prime }(r)}=\beta ,
\end{equation*}%
situated in $(0,\delta _{\nu ,1}).$

\textbf{c)} Observe that 
\begin{equation*}
1+\frac{zh_{\nu }^{\prime \prime }(z)}{h_{\nu }^{\prime }(z)}=1+\frac{%
zN_{\nu}^{\prime \prime }(\sqrt{z})+\left(3-2\nu\right)\sqrt{z}%
N_{\nu}^{\prime }(\sqrt{z})+\left(\nu^2-\nu\right)N_{\nu}(\sqrt{z})}{2\sqrt{z%
}N_{\nu}^{\prime }(\sqrt{z})+2\left(2-\nu\right)N_{\nu}(\sqrt{z})}.
\end{equation*}

By using (\ref{F1}) and (\ref{r11}) we have that%
\begin{eqnarray}
h_{\nu }^{\prime }(z) &=&\frac{2^{\nu -1}\Gamma (\nu +1)z^{\frac{-\nu }{2}}}{%
Q\left(\nu \right)} \left[ (2-\nu )N_{\nu }(\sqrt{z})+\sqrt{z}N_{\nu
}^{\prime}(\sqrt{z})\right]  \label{r32} \\
&=&\sum_{n=0}^{\infty }\frac{\left( -1\right) ^{n}(n+1)\Gamma (\nu
+1)Q\left( 2n+\nu \right)}{n!\Gamma (n+\nu +1)Q\left(\nu \right)}\left( 
\frac{z }{4}\right) ^{n}  \notag
\end{eqnarray}%
and 
\begin{equation*}
\lim\limits_{n\longrightarrow \infty}\frac{n \log n}{%
\begin{array}{c}
\lbrack n\log4+\log\left[\Gamma\left(n+\nu+1 \right) \right]+\log\left[%
\Gamma\left(n+1 \right) \right]+\log\left[Q\left(\nu\right)\right] \\ 
-\log\left[Q\left(2n+\nu\right) \right]-\log\left(n+1\right)-\log\left[%
\Gamma\left(\nu+1 \right) \right]%
\end{array}%
} = \frac{1}{2} .
\end{equation*}%
So, by applying Hadamard's Theorem \cite[p. 26]{Le} we can write the
infinite product representation of $h_{\nu }^{\prime }(z)$ as follows:%
\begin{equation}
h_{\nu }^{\prime }(z)=\prod\limits_{n\geq 1}\left( 1-\frac{z}{\gamma _{\nu
,n}^2}\right) ,  \label{r41}
\end{equation}%
where $\gamma _{\nu }$ denotes the $n$th positive zero of the function $%
h_{\nu ,n}^{\prime }.$ From Lemma \ref{Le4} for $\nu>\max\left\{0,\nu_0
\right\}$ the function $h_{\nu }^{\prime }\in \mathcal{LP}$, and the
smallest positive zero of $h_{\nu }^{\prime }$ does not exceed the first
positive zero of $N_{\nu }.$

By means of (\ref{r4}) we have 
\begin{equation*}
1+\frac{zh_{\nu }^{\prime \prime }(z)}{h_{\nu }^{\prime }(z)}%
=1-\sum\limits_{n\geq 1}\frac{z}{\gamma _{\nu ,n}^{2}-z}.
\end{equation*}%
By using the inequality (\ref{r2}), for all $z\in \mathbb{D}_{\gamma _{\nu
,n}}$ we obtain the inequality%
\begin{equation*}
\operatorname{Re}\left( 1+\frac{zh_{\nu }^{\prime \prime }(z)}{h_{\nu }^{\prime }(z)}%
\right) \geq 1-\sum\limits_{n\geq 1}\frac{r}{\gamma_{\nu ,n}^{2}-r},
\end{equation*}%
where $\left\vert z\right\vert =r.$ Thus, for $r\in (0,\gamma _{\nu ,1})$ we
get%
\begin{equation*}
\underset{z\in \mathbb{D}_{r}}{\inf }\left\{ \operatorname{Re}\left( 1+\frac{zh_{\nu
}^{\prime \prime }(z)}{h_{\nu }^{\prime }(z)}\right) \right\} =1+\frac{%
rh_{\nu}^{\prime \prime }(r)}{h_{\nu }^{\prime }(r)}.
\end{equation*}%
The function $\Upsilon_{\nu}:(0,\gamma _{\nu ,1})\rightarrow 
%TCIMACRO{\U{211d} }%
%BeginExpansion
\mathbb{R}
%EndExpansion
,$ defined by%
\begin{equation*}
\Upsilon_{\nu}(r)=1+\frac{rh_{\nu }^{\prime \prime }(r)}{h_{\nu }^{\prime
}(r)},
\end{equation*}%
is strictly decreasing and $\lim_{r\searrow 0}\Upsilon_{\nu }(r)=1>\beta $
and $\lim_{r\nearrow \gamma _{\nu ,1}}\Upsilon_{\nu }(r)=-\infty $.
Consequently, in view of the minimum principle for harmonic functions for $%
z\in \mathbb{D}_{r_{6}}$ we have that%
\begin{equation*}
\operatorname{Re}\left( 1+\frac{zh_{\nu }^{\prime \prime }(z)}{h_{\nu }^{\prime }(z)}%
\right) >\beta
\end{equation*}%
if and ony if $r_{6}$ is the unique root of 
\begin{equation*}
1+\frac{rh_{\nu }^{\prime \prime }(r)}{h_{\nu }^{\prime }(r)}=\beta ,
\end{equation*}%
situated in $(0,\gamma _{\nu ,1}).$
\end{proof}
\begin{eqnarray*}
&&%
\begin{tabular}{c|c|c||c|c|c||c|c|c|}
\cline{2-9}
& \multicolumn{8}{|c|}{$r_c\left( f\right) $} \\ \cline{2-9}
& \multicolumn{2}{|c||}{$b=1$ and $c=0$} &  & \multicolumn{2}{|c||}{$a=1$
and $c=0$} &  & \multicolumn{2}{|c|}{$a=1$ and $b=2$} \\ 
\cline{2-3}\cline{5-6}\cline{8-9}
& $\beta =0$ & $\beta =0.5$ &  & $\beta =0$ & $\beta =0.5$ &  & $\beta =0$ & 
$\beta =0.5$ \\ \hline
\multicolumn{1}{|c|}{$a=2$} & $1.0057$ & $0.7896$ & $b=2$ & $1.1515$ & $
0.8992$ & $c=2$ & $1.2412$ & $0.9653$ \\ \hline
\multicolumn{1}{|c|}{$a=3$} & $0.9810$ & $0.7709$ & $b=3$ & $1.2069$ & $
0.9408$ & $c=3$ & $1.2800$ & $0.9938$ \\ \hline
\multicolumn{1}{|c|}{$a=4$} & $0.9676$ & $0.7607$ & $b=4$ & $1.2460$ & $%
0.9702$ & $c=4$ & $1.3155$ & $1.0197$ \\ \hline
\end{tabular}
\\
&& \\
&&\text{\ \ \ \ \ \ \ \ \ \ \ \ \ \ \ \ \textbf{Table4.} Radii of
convexity for }f_{\nu }\text{ when }\nu =2.5
\end{eqnarray*}
\begin{eqnarray*}
&&%
\begin{tabular}{c|c|c||c|c|c||c|c|c|}
\cline{2-9}
& \multicolumn{8}{|c|}{$r_c\left( g\right) $} \\ \cline{2-9}
& \multicolumn{2}{|c||}{$b=1$ and $c=0$} &  & \multicolumn{2}{|c||}{$a=1$
and $c=0$} &  & \multicolumn{2}{|c|}{$a=1$ and $b=2$} \\ 
\cline{2-3}\cline{5-6}\cline{8-9}
& $\beta =0$ & $\beta =0.5$ &  & $\beta =0$ & $\beta =0.5$ &  & $\beta =0$ & 
$\beta =0.5$ \\ \hline
\multicolumn{1}{|c|}{$a=2$} & $0.6839$ & $0.5219$ & $b=2$ & $0.7769$ & $
0.5913$ & $c=2$ & $0.8325$ & $0.6323$ \\ \hline
\multicolumn{1}{|c|}{$a=3$} & $0.6680$ & $0.5101$ & $b=3$ & $0.8122$ & $
0.6176$ & $c=3$ & $0.8563$ & $0.6498$ \\ \hline
\multicolumn{1}{|c|}{$a=4$} & $0.6594$ & $0.5036$ & $b=4$ & $0.8371$ & $%
0.6361$ & $c=4$ & $0.8780$ & $0.6658$ \\ \hline
\end{tabular}
\\
&& \\
&&\text{\ \ \ \ \ \ \ \ \ \ \ \ \ \ \ \ \textbf{Table5.} Radii of
convexity for }g_{\nu }\text{ when }\nu =2.5
\end{eqnarray*}
\begin{eqnarray*}
&&%
\begin{tabular}{c|c|c||c|c|c||c|c|c|}
\cline{2-9}
& \multicolumn{8}{|c|}{$r_c\left( h\right) $} \\ \cline{2-9}
& \multicolumn{2}{|c||}{$b=1$ and $c=0$} &  & \multicolumn{2}{|c||}{$a=1$
and $c=0$} &  & \multicolumn{2}{|c|}{$a=1$ and $b=2$} \\ 
\cline{2-3}\cline{5-6}\cline{8-9}
& $\beta =0$ & $\beta =0.5$ &  & $\beta =0$ & $\beta =0.5$ &  & $\beta =0$ & 
$\beta =0.5$ \\ \hline
\multicolumn{1}{|c|}{$a=2$} & $1.4835$ & $1.0997$ & $b=2$ & $1.9725$ & $
1.4489$ & $c=2$ & $2.3191$ & $1.6906$ \\ \hline
\multicolumn{1}{|c|}{$a=3$} & $1.4080$ & $1.0453$ & $b=3$ & $2.1779$ & $
1.5946$ & $c=3$ & $2.4797$ & $1.8014$ \\ \hline
\multicolumn{1}{|c|}{$a=4$} & $1.3681$ & $1.0165$ & $b=4$ & $2.3289$ & $%
1.7016$ & $c=4$ & $2.6322$ & $1.9060$ \\ \hline
\end{tabular}
\\
&& \\
&&\text{\ \ \ \ \ \ \ \ \ \ \ \ \ \ \ \ \textbf{Table6.} Radii of
convexity for }h_{\nu }\text{ when }\nu =2.5
\end{eqnarray*}

In the tables above, convexity radii according to special values of $a,b,c\in \mathbb{R}$ showed a monotony similar to radii of starlikeness.

\subsection{Bounds for Radii of Starlikeness and Convexity of The Functions $%
g_{\protect\nu }$ and $h_{\protect\nu }$}

In this subsection we consider two different functions $g_{\nu }$ and $%
h_{\nu }$ which are normalized forms of the function $N_\nu$ given by (\ref%
{F1}) . Here firstly our aim is to show that the radii of univalence of
these functions correspond to the radii of starlikeness.
\newpage
\begin{theorem}
\label{t4} The following statements hold:
\begin{description}
\item[a)] If $\nu >\max\{0,\nu_0 \},$ then $r^{\ast }(g_{\nu})$ satisfies
the inequalities 
\begin{equation*}
r^{\ast }(g_{\nu })<\sqrt{\frac{2(\nu+1)Q(\nu )}{Q(\nu +2)},}
\end{equation*}
\begin{equation*}
2\sqrt{\frac{(\nu+1)Q(\nu)}{3Q(\nu +2)}}<r^{\ast }(g_{\nu })<2\sqrt{\frac{%
Q(\nu+2)}{\frac{3Q^2(\nu +2)}{(\nu+1)Q(\nu)}-\frac{5Q(\nu+2) }{3(\nu+2)}}}
\end{equation*}
and 
\begin{equation*}
\sqrt[4]{\frac{16(\nu+1)^2(\nu +2)Q^2(\nu)}{9(\nu +2)Q^2(\nu
+2)-5(\nu+1)Q(\nu)Q(\nu +4) }}< r^{\ast }(g_{\nu })
\end{equation*}
\begin{equation*}
<\sqrt{ \frac{8(\nu+1)(\nu+3)Q(\nu)\left[9(\nu+2)Q^2(\nu
+2)-5(\nu+1)Q(\nu)Q(\nu+4) \right]}{ 9(\nu+3)Q(\nu+2)\left[6(\nu+2)Q^2(\nu
+2)-5(\nu+1)^2Q(\nu)Q(\nu+4) \right]+7(\nu+1)^2Q^2(\nu)Q(\nu+6)} }.
\end{equation*}

\item[b)] If $\nu >\max\{0,\nu_0 \},$ then $r^{\ast }(h_{\nu })$ satisfies
the inequalities 
\begin{equation*}
r^{\ast }(h_{\nu })<\frac{2(\nu+1)Q(\nu)}{Q(\nu +2)},
\end{equation*}
\begin{equation*}
\frac{2(\nu+1)Q(\nu)}{3Q(\nu +2)}< r^{\ast }(h_{\nu })<\frac{2Q(\nu+2)}{%
\frac{Q^2(\nu +2)}{(\nu+1)Q(\nu)}-\frac{3Q(\nu+4) }{4(\nu+2)}}
\end{equation*}
and 
\begin{equation*}
\sqrt{ \frac{4(\nu+1)Q(\nu )}{ \frac{Q^2(\nu +2)}{(\nu+1)Q(\nu)}-\frac{%
3Q(\nu +4)}{4(\nu +2)}} }< r^{\ast }(h_{\nu })
\end{equation*}
\begin{equation*}
<\frac{4(\nu+3)Q(\nu)\left[-4(\nu+2)Q^2(\nu+2)+3(\nu+1)^2Q(\nu) Q(\nu+4) %
\right] }{(\nu+3)Q(\nu+2)\left[8(\nu+2)Q^2(\nu+2)-9(\nu+1)Q(\nu)Q(\nu+4) %
\right]+2(\nu+1)^2Q^2(\nu)Q(\nu+6)}
\end{equation*}
\end{description}
\end{theorem}

\begin{proof}
\textbf{a)} By using the first Rayleigh sum (\ref{S5}) and the implict
relation for $r^{\ast }(g_{\nu }),$ obtained by Kreyszing and Todd \cite{Kr}%
, we get for all $\nu >\max\{0,\nu_0 \}$ that%
\begin{equation*}
\frac{1}{\left( r^{\ast }(g_{\nu })\right) ^{2}}=\sum\limits_{n\geq 1}\frac{2%
}{ \lambda_{\nu ,n} ^{2}-\left( r^{\ast }(g_{\nu })\right) ^{2}}%
>\sum\limits_{n\geq 1}\frac{2}{\lambda_{\nu ,n}^{2}}=\frac{Q(\nu +2)}{%
2(\nu+1)Q(\nu)}.
\end{equation*}%
Now, by using the Euler-Rayleigh inequalities it is possible to have more
tight bounds for the radius of univalence (and starlikeness) $r^{\ast
}(g_{\nu})$. We define the function $\Psi _{\nu }(z)=g_{\nu }^{\prime }(z)$,
where $g_{\nu}^{\prime }$ defined by (\ref{r4}). Now, taking logarithmic
derivative of both sides of (\ref{r4}) for $\left\vert z\right\vert <\delta
_{\nu ,1}$ we have%
\begin{equation}
\frac{\Psi _{\nu }^{\prime }(z)}{\Psi _{\nu }(z)}=-\sum\limits_{n\geq 1}%
\frac{2z}{\delta _{\nu ,n} ^{2}-z^{2}}=-2\sum\limits_{n\geq
1}\sum\limits_{k\geq 0}\frac{1}{\left( \delta _{\nu ,n}\right) ^{2(k+1)}}%
z^{2k+1}=-2\sum\limits_{k\geq 0}\sigma _{k+1}z^{2k+1}  \label{p1}
\end{equation}%
where $\sigma _{k}=\sum_{n\geq 1}\left( \delta _{\nu ,n}\right) ^{-k}$ is
Euler-Rayleigh sum for the zeros of $\Psi _{\nu }.$ Also, using (\ref{r31})
from the infinite sum representation of $\Psi _{\nu}$ we obtain%
\begin{equation}
\frac{\Psi _{\nu }^{\prime }(z)}{\Psi _{\nu}(z)}=\frac{\sum\limits_{n\geq
0}U_{n}z^{2n+1}}{\sum\limits_{n\geq 0}V_{n}z^{2n}},  \label{p2}
\end{equation}%
where 
\begin{equation*}
U_{n}=\frac{2\left( -1\right) ^{n+1}\Gamma (\nu+1)Q(2n+\nu +2)(2n+3) }{%
n!4^{n+1}\Gamma (n+\nu +2)Q(\nu)}
\end{equation*}%
and%
\begin{equation*}
V_{n}=\frac{\left( -1\right) ^{n}\Gamma (\nu +1)Q(2n+\nu)(2n+1)}{
n!4^{n}\Gamma (n+\nu +1)Q(\nu)}.
\end{equation*}%
By comparing the coefficients with the same degrees of (\ref{p1}) and (\ref%
{p2}) we obtain the Euler-Rayleigh sums%
\begin{equation*}
\sigma _{1}=\frac{3Q(\nu +2)}{4(\nu+1)Q(\nu)} \text{ and } \sigma _{2}= 
\frac{9(\nu +2)Q^2(\nu +2)-5(\nu+1)Q(\nu)Q(\nu +4)}{ 16(\nu+1)^2(\nu
+2)Q^2(\nu)}
\end{equation*}
and 
\begin{equation*}
\sigma _{3}= \frac{54(\nu+2)(\nu+3)Q^3(\nu
+2)-45(\nu+1)^2(\nu+3)Q(\nu)Q(\nu+2)Q(\nu+4)+7(\nu+1)^2Q^2(\nu)Q(\nu+6)}{
128(\nu+1)^3(\nu+2)(\nu+3)Q^3(\nu)}
\end{equation*}%
By using the Euler-Rayleigh inequalities%
\begin{equation*}
\sigma _{k}^{-\frac{1}{k}}<\delta _{\nu ,n} ^{2}<\frac{\sigma _{k}}{\sigma
_{k+1}}
\end{equation*}%
for $\nu >\max\{0,\nu_0 \}$, $k\in 
%TCIMACRO{\U{2115} }%
%BeginExpansion
\mathbb{N}
%EndExpansion
$ and for $k=1$ and $k=2$ we get the following inequalities 
\begin{equation*}
\frac{4(\nu+1)Q(\nu)}{3Q(\nu +2)}<\left( r^{\ast }(g_{\nu })\right) ^{2}<%
\frac{4Q(\nu+2)}{\frac{3Q^2(\nu +2)}{(\nu+1)Q(\nu)}-\frac{5Q(\nu+4) }{%
3(\nu+2)}}
\end{equation*}
and 
\begin{equation*}
\sqrt{\frac{16(\nu+1)^2(\nu +2)Q^2(\nu)}{9(\nu +2)Q^2(\nu
+2)-5(\nu+1)Q(\nu)Q(\nu +4) }}<\left( r^{\ast }(g_{\nu })\right) ^{2}<
\end{equation*}
\begin{equation*}
< \frac{8(\nu+1)(\nu+3)Q(\nu)\left[9(\nu+2)Q^2(\nu
+2)-5(\nu+1)Q(\nu)Q(\nu+4) \right]}{ 9(\nu+3)Q(\nu+2)\left[6(\nu+2)Q^2(\nu
+2)-5(\nu+1)^2Q(\nu)Q(\nu+4) \right]+7(\nu+1)^2Q^2(\nu)Q(\nu+6)}
\end{equation*}
and it is possible to have more tighter bounds for other values of $k\in 
%TCIMACRO{\U{2115} }%
%BeginExpansion
\mathbb{N}
%EndExpansion
.$

\textbf{b)} By using the first Rayleigh sum (\ref{S5}) and the implict
relation for $r^{\ast }(h_{\nu}),$ obtained by Kreyszing and Todd \cite{Kr},
we get for all $\nu >\max\{0,\nu_0 \}$ that 
\begin{equation*}
\frac{1}{r^{\ast }(h_{\nu})}=\sum\limits_{n\geq 1}\frac{1}{ \lambda_{\nu
,n}^{2}-r^{\ast }(h_{\nu})}>\sum\limits_{n\geq 1}\frac{1}{\lambda_{\nu ,n}
^{2}}=\frac{Q(\nu +2)}{2(\nu+1)Q(\nu)}.
\end{equation*}%
Now, by using the Euler-Rayleigh inequalities it is possible to have more
tight bounds for the radius of univalence (and starlikeness) $r^{\ast
}(h_{\nu })$. We define the function $\Phi _{\nu }(z)=h_{\nu}^{\prime}(z)$,
where $h_{\nu }^{\prime }$ defined by (\ref{r32}) or (\ref{r41}). Now,
taking logarithmic derivative of both sides of (\ref{r41}) we have%
\begin{equation}
\frac{\Phi _{\nu}^{\prime }(z)}{\Phi _{\nu}(z)}=-\sum\limits_{n\geq 1}\frac{1%
}{\gamma _{\nu ,n}-z}=-\sum\limits_{n\geq 1}\sum\limits_{k\geq 0}\frac{1}{%
\left( \gamma _{\nu ,n}\right) ^{k+1}}z^{k}=-\sum\limits_{k\geq 0}\rho
_{k+1}z^{k},\text{ \ \ }\left\vert z\right\vert <\gamma _{\nu ,1}
\label{p11}
\end{equation}%
where $\rho _{k}=\sum_{n\geq 1}\left( \gamma _{\nu ,n}\right) ^{-k}$ is
Euler-Rayleigh sum for the zeros of $\Phi _{\nu }.$ Also, using (\ref{r32})
from the infinite sum representation of $\Phi _{\nu }$ we obtain%
\begin{equation}
\frac{\Phi _{\nu }^{\prime }(z)}{\Phi _{\nu }(z)}=\frac{\sum\limits_{n\geq
0}K_{n}z^{n}}{\sum\limits_{n\geq 0}L_{n}z^{n}},  \label{p12}
\end{equation}%
where%
\begin{equation*}
K_{n}=\frac{\left( -1\right) ^{n+1}\Gamma (\nu +1)Q(2n+\nu+2)(n+2)}{
n!4^{n+1}\Gamma (n+\nu +2)Q(\nu)}
\end{equation*}%
and%
\begin{equation*}
L_{n}=\frac{\left( -1\right) ^{n}\Gamma (\nu +1)Q(2n+\nu)(n+1)}{%
n!4^{n}\Gamma (n+\nu +1)Q(\nu)}.
\end{equation*}%
By comparing the coefficients with the same degrees of (\ref{p11}) and (\ref%
{p12}) we obtain the Euler-Rayleigh sums 
\begin{equation*}
\rho _{1}=\frac{Q(\nu +2)}{2(\nu+1)Q(\nu)} \text{ and } \rho _{2}=\frac{1}{%
4(\nu+1)Q(\nu )}\left( \frac{Q^2(\nu +2)}{(\nu+1)Q(\nu)}-\frac{3Q(\nu +4)}{%
4(\nu +2)}\right)
\end{equation*}%
and 
\begin{equation*}
\rho _{3}=\frac{1}{64(\nu+1)^3Q^3(\nu )}\left(8Q^3(\nu+2)- \frac{%
9(\nu+1)Q(\nu)Q(\nu +2)Q(\nu +4)}{(\nu+2)}+\frac{2(\nu+1)^2Q^2(\nu)Q(\nu +6)%
}{(\nu +2)(\nu +3)} \right)
\end{equation*}%
By using the Euler-Rayleigh inequalities 
\begin{equation*}
\rho _{k}^{-\frac{1}{k}}<\gamma _{\nu ,n}<\frac{\rho _{k}}{\rho _{k+1}}
\end{equation*}%
for $\nu >\max\{0,\nu_0 \}$, $k\in 
%TCIMACRO{\U{2115} }%
%BeginExpansion
\mathbb{N}
%EndExpansion
$ and for $k=1$ and $k=2,$ we get the following inequalities 
\begin{equation*}
\frac{2(\nu+1)Q(\nu)}{3Q(\nu +2)}< r^{\ast }(h_{\nu })<\frac{2Q(\nu+2)}{%
\frac{Q^2(\nu +2)}{(\nu+1)Q(\nu)}-\frac{3Q(\nu+4) }{4(\nu+2)}}
\end{equation*}%
and 
\begin{equation*}
\sqrt{ \frac{4(\nu+1)Q(\nu )}{ \frac{Q^2(\nu +2)}{(\nu+1)Q(\nu)}-\frac{%
3Q(\nu +4)}{4(\nu +2)}} }< r^{\ast }(h_{\nu })<
\end{equation*}
\begin{equation*}
<\frac{4(\nu+3)Q(\nu)\left[-4(\nu+2)Q^2(\nu+2)+3(\nu+1)^2Q(\nu) Q(\nu+4) %
\right] }{(\nu+3)Q(\nu+2)\left[8(\nu+2)Q^2(\nu+2)-9(\nu+1)Q(\nu)Q(\nu+4) %
\right]+2(\nu+1)^2Q^2(\nu)Q(\nu+6)}
\end{equation*}
and it is possible to have more tighter bounds for other values of $k\in 
%TCIMACRO{\U{2115} }%
%BeginExpansion
\mathbb{N}
%EndExpansion
.$
\end{proof}

The next result concerning bounds for radii of convexity of functions $%
g_{\nu }$ and $h_{\nu }$.

\begin{theorem}
\label{t5} The following statements hold:

\begin{description}
\item[a)] If $\nu\geq \max\{0,\nu_0\}$ then $r^{c}(g_{\nu })$ satisfies the
inequalities 
\begin{equation*}
\frac{2}{3}\sqrt{\frac{(\nu +1)Q(\nu)}{Q(\nu +2)}}<r^{c}(g_{\nu })< \sqrt{ 
\frac{36(\nu+1)(\nu+2)Q(\nu+2)}{81(\nu+2)Q^2(\nu+2)-25(\nu+1)Q(\nu)Q(\nu+4)} 
}
\end{equation*}
and 
\begin{equation*}
\sqrt[4]{\frac{16(\nu+1)^2(\nu+2)Q(\nu)}{81(\nu+2)Q^2(\nu+2)-25(\nu+1)Q(%
\nu)Q(\nu+4)}}< r^{c}(g_{\nu})
\end{equation*}
\begin{equation*}
<\sqrt{\frac{8(\nu+1)(\nu+3)Q(\nu)\left[81(\nu+2)Q^2(\nu+2)-25(\nu+1)Q(%
\nu)Q(\nu+4)\right]}{27(\nu+3)Q(\nu+2)\left[54(\nu+2)Q^2(\nu+2)-25(%
\nu+1)^2Q(\nu)Q(\nu+4)\right]+49(\nu+1)^2Q^2(\nu)Q(\nu+6) }}.
\end{equation*}

\item[b)] If $\nu\geq \max\{0,\nu_0\}$ then $r^{c}(h_{\nu })$ satisfies the
inequalities 
\begin{equation*}
\frac{(\nu +1)Q(\nu)}{Q(\nu +2)}<r^{c}(h_{\nu })<\frac{16(\nu+2)Q(\nu)Q(%
\nu+2)}{16(\nu+2)Q^2(\nu+2)-9Q(\nu)Q(\nu+4)}
\end{equation*}
and 
\begin{equation*}
\sqrt{\frac{16(\nu+1)(\nu+2)Q^2(\nu)}{16(\nu+2)Q^2(\nu+2)-9Q(\nu)Q(\nu+4)}}%
<r^{c}(h_{\nu })
\end{equation*}
\begin{equation*}
<\frac{2(\nu+1)^2(\nu+3)Q(\nu)\left[16(\nu+2)Q^2(\nu+2)-9Q(\nu) Q(\nu+4) %
\right] }{(\nu+3)Q(\nu+2)\left[32(\nu+2)Q^2(\nu+2)-27(\nu+1)^2Q(\nu)Q(\nu+4) %
\right]+4(\nu+1)^2Q^2(\nu)Q(\nu+6)}.
\end{equation*}
\end{description}
\end{theorem}

\begin{proof}
\textbf{a)} By using the Alexander duality theorem for starlike and convex
functions we can say that the function $g_{\nu }(z)$ is convex if and only
if $zg_{\nu }^{\prime }(z)$ is starlike. But, the smallest positive zero of $%
z\mapsto z\left( zg_{\nu }^{\prime }(z)\right) ^{\prime }$ is actually the
radius of starlikeness of $z\mapsto \left( zg_{\nu }^{\prime }(z)\right) $,
according to Theorem \ref{t2} and Theorem \ref{t3}. Therefore, the radius of
convexity $r^{c}(g_{\nu})$ is the smallest positive root of the equation $%
\left( zg_{\nu }^{\prime }(z)\right) ^{\prime }=0$. Therefore from (\ref{r31}%
), we have%
\begin{equation*}
\Delta _{\nu }(z)=\left( zg_{\nu }^{\prime }(z)\right) ^{\prime
}=\sum_{n=0}^{\infty }\frac{\left( -1\right) ^{n}(2n+1)^2Q(2n+\nu)\Gamma
(\nu +1)}{n!\Gamma (n+\nu +1)Q(\nu)}\left( \frac{z}{2}\right) ^{2n}
\end{equation*}
Since the function $g_{\nu }(z)$ belongs to the Laguerre-P\'{o}lya class of
entire functions and $\mathcal{LP}$ is closed under differentiation, we can
say that the function $\Delta _{\nu }(z)\in \mathcal{LP}$. Therefore, the
zeros of the function $\Delta _{\nu}$ are all real. Suppose that $d_{\nu ,n}$
are the zeros of the function $\Delta _{\nu }$. Then the function $\Delta
_{\nu }$ has the infinite product representation as follows: 
\begin{equation}
\Delta _{\nu }(z)=\prod\limits_{n\geq 1}\left( 1-\frac{z^{2}}{ d_{\nu
,n}^{2} }\right) .  \label{p13}
\end{equation}%
By taking the logarithmic derivative of (\ref{p13}) we get%
\begin{equation}
\frac{\Delta _{\nu}^{\prime }(z)}{\Delta _{\nu }(z)}=-2\sum\limits_{n\geq 1}%
\frac{z}{d_{\nu ,n}^{2}-z^{2}} =-2\sum\limits_{n\geq 1}\sum\limits_{k\geq 0}%
\frac{1}{\left( d_{\nu ,n}\right) ^{2(k+1)}}z^{2k+1}=-2\sum\limits_{k\geq
0}\kappa _{k+1}z^{2k+1},~\left\vert z\right\vert <d_{\nu ,1}  \label{p14}
\end{equation}
where $\kappa _{k}=\sum_{n\geq 1}\left( d_{\nu ,n}\right) ^{-k}$ is
Euler-Rayleigh sum for the zeros of $\Delta _{\nu }.$ On the other hand, by
considering infinite sum representation of $\Delta _{\nu }(z)$ we obtain 
\begin{equation}
\frac{\Delta _{\nu }^{\prime }(z)}{\Delta _{\nu }(z)}=\frac{%
\sum\limits_{n\geq 0}X_{n}z^{2n+1}}{\sum\limits_{n\geq 0}Y_{n}z^{2n}},
\label{p15}
\end{equation}%
where 
\begin{equation*}
X_{n}=\frac{2\left( -1\right) ^{n+1}\Gamma (\nu+1)Q(2n+\nu +2)(2n+3)^2 }{%
n!4^{n+1}\Gamma (n+\nu +2)Q(\nu)}
\end{equation*}%
and%
\begin{equation*}
Y_{n}=\frac{\left( -1\right) ^{n}\Gamma (\nu +1)Q(2n+\nu)(2n+1)^2}{
n!4^{n}\Gamma (n+\nu +1)Q(\nu)}.
\end{equation*}%
By comparing the coefficients of (\ref{p14}) and (\ref{p15}) we obtain%
\begin{equation*}
\kappa _{1}=\frac{9Q(\nu +2)}{4(\nu +1)Q(\nu)} \text{ and } \kappa _{2}=%
\frac{81(\nu+2)Q^2(\nu+2)-25(\nu+1)Q(\nu)Q(\nu+4)}{16(\nu+1)^2(\nu+2)Q(\nu)}
\end{equation*}%
and%
\begin{equation*}
\kappa _{3}=\frac{27(\nu+3)Q(\nu+2)\left[54(\nu+2)Q^2(\nu+2)-25(\nu+1)^2Q(%
\nu)Q(\nu+4)\right]+49(\nu+1)^2Q^2(\nu)Q(\nu+6)}{128(\nu+1)^3(\nu+2)(%
\nu+3)Q^3(\nu)}
\end{equation*}%
By using the Euler-Rayleigh inequalities 
\begin{equation*}
\kappa _{k}^{-\frac{1}{k}}<d_{\nu ,n}^2 <\frac{\kappa _{k}}{\kappa _{k+1}}
\end{equation*}
for $\nu \geq\max\{0,\nu_0 \}$, $k\in \mathbb{N} $ and for $k=1$ and $h=2,$
we get the following inequalities 
\begin{equation*}
\frac{4(\nu +1)Q(\nu)}{9Q(\nu +2)}<\left( r^{c}(g_{\nu })\right) ^{2}< \frac{%
36(\nu+1)(\nu+2)Q(\nu)Q(\nu+2)}{81(\nu+2)Q^2(\nu+2)-25(\nu+1)Q(\nu)Q(\nu+4)}
\end{equation*}
and 
\begin{equation*}
\frac{4}{9}\sqrt{\frac{(\nu+1)^2(\nu+2)Q(\nu)}{(\nu+2)Q^2(\nu+2)-25(\nu+1)Q(%
\nu)Q(\nu+4)}}<\left( r^{c}(g_{\nu})\right) ^{2}
\end{equation*}
\begin{equation*}
<\frac{8(\nu+1)(\nu+3)Q(\nu)\left[81(\nu+2)Q^2(\nu+2)-25(\nu+1)Q(\nu)Q(\nu+4)%
\right]}{27(\nu+3)Q(\nu+2)\left[54(\nu+2)Q^2(\nu+2)-25(\nu+1)^2Q(\nu)Q(\nu+4)%
\right]+49(\nu+1)^2Q^2(\nu)Q(\nu+6) }
\end{equation*}
\textbf{b)} By using the same procedure as in the previous proof we can say
that the radius of convexity $r^{c}(h_{\nu})$ is the smallest positive root
of the equation $\left( zh_{\nu }^{\prime }(z)\right) ^{\prime }=0$
according to Theorem \ref{t3}. From (\ref{r32}), we have%
\begin{equation}
\Theta _{\nu}(z)=\left( zh_{\nu}^{\prime }(z)\right) ^{\prime
}=\sum_{n=0}^{\infty }\frac{\left( -1\right) ^{n}(n+1)^2\Gamma (\nu
+1)Q\left( 2n+\nu \right)}{n!\Gamma (n+\nu +1)Q\left(\nu \right)}\left( 
\frac{z }{4}\right) ^{n}.  \label{p151}
\end{equation}%
Moreover, we know $h_{\nu}(z)$ belongs to the Laguerre-P\'{o}lya class of
entire functions and $\mathcal{LP}$, consequently $\Theta _{\nu }(z)\in 
\mathcal{LP}$. On the other words, the zeros of the function $\Theta _{\nu}$
are all real. Assume that $l_{\nu ,n}$ are the zeros of the function $\Theta
_{\nu}$. In this case, the function $\Theta _{\nu}$ has the infinite product
representation as follows: 
\begin{equation}
\Theta _{\nu}(z)=\prod\limits_{n\geq 1}\left( 1-\frac{z^{2}}{\ l_{\nu ,n}^2}%
\right) .  \label{p16}
\end{equation}%
By taking the logarithmic derivative of both sides of (\ref{p16}) for \ $%
\left\vert z\right\vert <l_{\nu ,1}$ we have%
\begin{equation}
\frac{\Theta _{\nu}^{\prime }(z)}{\Theta _{\nu}(z)}=-\sum\limits_{n \geq 1}%
\frac{1}{l_{\nu ,n}-z}=-\sum\limits_{n\geq 1}\sum\limits_{k\geq 0}\frac{1}{%
\left( l_{\nu ,n}\right) ^{k+1}}z^{k}=-\sum\limits_{k\geq 0}\omega
_{k+1}z^{k}  \label{p17}
\end{equation}%
where $\omega _{k}=\sum_{n\geq 1}\left( l_{\nu ,n}\right) ^{-k}$. In
addition, by using the derivative of infinite sum representation considering
infinite sum representation of (\ref{p151}) we obtain%
\begin{equation}
\frac{\Theta _{\nu }^{\prime }(z)}{\Theta _{\nu}(z)}=\sum\limits_{n \geq
0}T_{n}z^{n}\diagup \sum\limits_{n\geq 0}S_{n}z^{n},  \label{p18}
\end{equation}
where%
\begin{equation*}
T_{n}=\frac{\left( -1\right) ^{n+1}(n+2)^2\Gamma (\nu +1)Q\left( 2n+\nu+2
\right)}{n!4^{n+1}\Gamma (n+\nu +2)Q\left(\nu \right)}
\end{equation*}%
and%
\begin{equation*}
S_{n}=\frac{\left( -1\right) ^{n}(n+1)^2\Gamma (\nu +1)Q\left( 2n+\nu \right)%
}{n!4^n\Gamma (n+\nu +1)Q\left(\nu \right)}.
\end{equation*}%
By comparing the coefficients of (\ref{p17}) and (\ref{p18}) we get 
\begin{equation*}
\omega _{1}=\frac{Q(\nu +2)}{(\nu +1)Q(\nu)} \text{ and } \omega _{2}=\frac{%
16(\nu+2)Q^2(\nu+2)-9Q(\nu)Q(\nu+4)}{16(\nu+1)(\nu+2)Q^2(\nu)}
\end{equation*}%
and%
\begin{equation*}
\omega _{3}=\frac{1}{32(\nu+1)^3Q^3(\nu )}\left(32Q^3(\nu+2)- \frac{%
27(\nu+1)^2Q(\nu)Q(\nu +2)Q(\nu +4)}{(\nu+2)}+\frac{4(\nu+1)^2Q^2(\nu)Q(\nu
+6)}{(\nu +2)(\nu +3)} \right)
\end{equation*}
By using the Euler-Rayleigh inequalities%
\begin{equation*}
\omega _{k}^{-\frac{1}{k}}<l_{\nu ,n}<\frac{\omega _{k}}{\omega _{k+1}}
\end{equation*}%
for $\nu \geq\max\{0,\nu_0 \}$, $k\in \mathbb{N}$ and for $k=1$ and $k=2,$
we get the following inequality 
\begin{equation*}
\frac{(\nu +1)Q(\nu)}{Q(\nu +2)}<r^{c}(h_{\nu })<\frac{16(\nu+2)Q(\nu)Q(%
\nu+2)}{16(\nu+2)Q^2(\nu+2)-9Q(\nu)Q(\nu+4)}
\end{equation*}
and 
\begin{equation*}
\sqrt{\frac{16(\nu+1)(\nu+2)Q^2(\nu)}{16(\nu+2)Q^2(\nu+2)-9Q(\nu)Q(\nu+4)}}%
<r^{c}(h_{\nu })
\end{equation*}
\begin{equation*}
<\frac{2(\nu+1)^2(\nu+3)Q(\nu)\left[16(\nu+2)Q^2(\nu+2)-9Q(\nu) Q(\nu+4) %
\right] }{(\nu+3)Q(\nu+2)\left[32(\nu+2)Q^2(\nu+2)-27(\nu+1)^2Q(\nu)Q(\nu+4) %
\right]+4(\nu+1)^2Q^2(\nu)Q(\nu+6)}
\end{equation*}
and it is possible to have more tighter bounds for other values of $k\in 
\mathbb{N} .$
\end{proof}

\subsection{Appendix}
We know  the $\Gamma(z)$ is defined by
\begin{equation*}
\frac{1}{\Gamma(z)}=z\prod_{j=1}^{\infty}\left(1+\frac{z}{j} \right)\left(1+%
\frac{1}{j} \right)^{-z},\text{ \ \ \ \ \ } z\in \mathbb{C}.
\end{equation*}
From \ref{J2} and \ref{r11}, we have 
\begin{equation*}
\frac{2^{\nu }\Gamma (\nu +1)N_\nu(z^{1/2})}{Q(\nu )z^{{\nu/2} }}%
=\prod\limits_{n\geq 1}\left( 1-\frac{z}{ \lambda_{\nu ,n} ^{2}}%
\right)=\sum_{n=0}^{\infty }\frac{\left( -1\right)
^{n}\Gamma\left(\nu+1\right)Q\left(2n+\nu\right)z^n}{n!4^n\Gamma (n+\nu
+1)Q\left(\nu\right)}.
\end{equation*}
Thus, in \cite{Zhang2010} using the Lemma 1, we obtain 
\begin{equation*}
s_n=\sum_{k=1}^{\infty}\frac{1}{\lambda_{\nu ,k}^{2n}} 
\end{equation*}
and 
\begin{equation*}
\sum_{1\leq k_1<k_2<\cdots<k_n} \frac{1}{\lambda_{\nu
,k_1}^2\cdot\lambda_{\nu ,k_2}^2\cdots\lambda_{\nu ,k_n}^2}=\frac{%
\Gamma\left(\nu+1\right)Q\left(2n+\nu\right)}{ n!4^n\Gamma (n+\nu
+1)Q\left(\nu\right)} 
\end{equation*}
or 
\begin{equation*}
\sum_{1\leq k_1<k_2<\cdots<k_n} \frac{1}{\lambda_{\nu
,k_1}^2\cdot\lambda_{\nu ,k_2}^2\cdots\lambda_{\nu ,k_n}^2}=\frac{%
Q\left(2n+\nu\right)}{ n!4^n(\nu +1)_nQ\left(\nu\right)} .
\end{equation*}
\newline
Then, again apply in the Lemma 1 in \cite{Zhang2010}, with $c=-\frac{1}{4}$ to get 
\newline
$\sum_{k=1}^{\infty}\frac{4^n\left(-1\right)^n}{\lambda_{\nu ,k}^{2n}}$%
\newline
$=\text{det}$ $\left( 
\begin{array}{cccccc}
1 & 0 & 0 & \cdots & 0 & \frac{-Q\left( \nu +2\right) }{\left( \nu +1\right)
_{1}Q\left( \nu \right)} \\ 
\frac{Q\left( \nu +2\right) }{ \left( \nu +1\right) _{1}Q\left( \nu \right)}
& 1 & 0 & \cdots & 0 & \frac{-Q\left( \nu +4\right) }{ \left( \nu +1\right)
_{2}Q\left( \nu \right)} \\ 
\frac{Q\left( \nu +4\right) }{2! \left( \nu +1\right) _{2}Q\left( \nu
\right) } & \frac{Q\left( \nu +2\right) }{ \left( \nu +1\right) _{1}Q\left(
\nu \right)} & 1 & \cdots & 0 & \frac{-Q\left( \nu +6\right) }{2! \left( \nu
+1\right) _{3}Q\left( \nu \right)} \\ 
\vdots & \vdots & \vdots & \ddots & \vdots & \vdots \\ 
\frac{Q\left( \nu +2\left( n-2\right) \right) }{\left( n-2\right) ! \left(
\nu +1\right) _{n-2}Q\left( \nu \right)} & \frac{Q\left( \nu +2\left(
n-3\right) \right) }{\left( n-3\right) ! \left( \nu +1\right) _{n-3}Q\left(
\nu \right)} & \frac{Q\left( \nu +2\left( n-4\right) \right) }{\left(
n-4\right) ! \left( \nu +1\right) _{n-4}Q\left( \nu \right)} & \cdots & 1 & 
\frac{-Q\left( \nu +2\left( n-1\right) \right) }{\left( n-2\right) ! \left(
\nu +1\right) _{n-1}Q\left( \nu \right)} \\ 
\frac{Q\left( \nu +2\left( n-1\right) \right) }{\left( n-1\right) ! \left(
\nu +1\right) _{n-1}Q\left( \nu \right)} & \frac{Q\left( \nu +2\left(
n-2\right) \right) }{\left( n-2\right) ! \left( \nu +1\right) _{n-2}Q\left(
\nu \right)} & \frac{Q\left( \nu +2\left( n-3\right) \right) }{\left(
n-3\right) ! \left( \nu +1\right) _{n-3}Q\left( \nu \right)} & \cdots & 
\frac{Q\left( \nu +2\right) }{ \left( \nu +1\right) _{1}Q\left( \nu \right)}
& \frac{-Q\left( \nu +2n\right) }{\left( n-1\right) ! \left( \nu +1\right)
_{n}Q\left( \nu \right)}
\end{array}%
\right) .$ Here are the first few $s_n,$ 
\begin{eqnarray*}
s_1 &=& \frac{Q\left(\nu+2\right) }{4\left(\nu+1\right)_1Q\left(\nu\right) }
\\
s_2&=& \frac{\left(\nu+2\right)Q^2\left(\nu+2\right)-\left(\nu+1\right)Q%
\left(\nu\right)Q\left(\nu+4\right)}{4^2Q^2\left(\nu\right)\prod_{j=1}^{2}%
\left(\nu+1\right)_j} \\
s_3&=& \frac{\left[\begin{array}{l}\left(\nu+2\right)_2Q\left(\nu+2\right)\left[
2\left(\nu+2\right)Q^2\left(\nu+2\right)-3\left(\nu+1\right)Q\left(\nu%
\right)Q\left(\nu+4\right)\right]\\+\left(\nu+1\right)^2\left(\nu+2\right)Q^2%
\left(\nu\right)Q\left(\nu+6\right)\end{array}\right] }{2!4^3Q^3\left(\nu\right)\prod_{j=1}^{3}%
\left(\nu+1\right)_j} \\
s_4&=& \frac{ \left[
\begin{array}{ll}
\left(\nu+2\right)_2\lbrack
6\left(\nu+2\right)\left(\nu+2\right)_3Q^4\left(\nu+2\right)-12\left(\nu+1%
\right)_4Q\left(\nu\right)Q^2\left(\nu+2\right)Q\left(\nu+4\right) &  \\ 
+\left(\nu+1\right)\left(\nu+1\right)_2\left(\nu+4\right)Q^2\left(\nu%
\right)Q\left(\nu+2\right)\left(Q\left(\nu+4\right)
+3Q\left(\nu+6\right)\right) &  \\ 
+\left(\nu+1\right)^2Q^2\left(\nu\right)\left(3\left(\nu+2\right)_2Q^2\left(%
\nu+4\right)-\left(\nu+1\right)_2Q\left(\nu\right)Q\left(\nu+8\right)\right)
\rbrack & 
\end{array}\right]
}{3!4^4Q^4\left(\nu\right)\prod_{j=1}^{4}\left(\nu+1\right)_j}
\end{eqnarray*}

\end{document}